\documentclass[11pt]{article}

\usepackage[a4paper,left=2.5cm,right=2.5cm,top=2.8cm,bottom=2.8cm]{geometry}
\usepackage[T1]{fontenc}
\usepackage[utf8]{inputenc}
\usepackage{lmodern,microtype,xcolor,amsmath,amsthm,amssymb}
\usepackage{empheq,mathtools,enumitem,etoolbox}
\usepackage{physics}
\usepackage{dsfont}
\usepackage[colorlinks=true,linkcolor=black,urlcolor=black,citecolor=blue]{hyperref}
\hypersetup{pdftitle={Second quantization and compactness of induced Gaussian quantum Markov semigroups},pdfauthor={Zheng Li},pdfkeywords={Gaussian quantum Markov semigroup, compact resolvent, second quantization, GNS embedding, trace class}}
\numberwithin{equation}{section}
\newcommand{\C}{\mathbb{C}}
\newcommand{\N}{\mathbb{N}}
\newcommand{\R}{\mathbb{R}}
\newcommand{\1}{\mathds{1}}
\newcommand{\cB}{\mathcal{B}}
\newcommand{\cH}{\mathcal{H}}
\newcommand{\cL}{\mathcal{L}}
\newcommand{\cS}{\mathcal{S}}
\newcommand{\cT}{\mathcal{T}}
\newcommand{\mi}{\mathrm{i}}
\newcommand{\me}{\mathrm{e}}
\newcommand{\md}{\mathrm{d}}
\DeclareMathOperator{\dom}{Dom}
\DeclareMathOperator{\myspan}{span}
\DeclareMathOperator{\diag}{diag}
\DeclareMathOperator{\gap}{gap}
\newtheorem{theorem}{Theorem}[section]
\newtheorem{proposition}[theorem]{Proposition}
\newtheorem{lemma}[theorem]{Lemma}
\newtheorem{corollary}[theorem]{Corollary}
\theoremstyle{definition}
\newtheorem{example}[theorem]{Example}
\newtheorem{remark}[theorem]{Remark}
\allowdisplaybreaks
\makeatletter
\AtBeginDocument{\patchcmd{\@citex}{,\penalty\@m\ }{;\penalty\@m\ }{}{}}
\patchcmd{\thebibliography}{\advance\leftmargin\labelsep}
 {\setlength{\labelsep}{1em}\advance\leftmargin\labelsep}{}{}
\makeatother

\title{Second quantization and compactness of induced Gaussian quantum Markov semigroups}
\author{Zheng Li\\[1ex]
\small School of Mathematics and Statistics, Central South University\\
\small Changsha 410083, China\\[0.5ex]
\small\href{mailto:zheng.li@csu.edu.cn}{zheng.li@csu.edu.cn}}
\date{}

\begin{document}
\maketitle
\begin{abstract}
We characterize compactness of the Hilbert--Schmidt semigroups induced by finite-mode Gaussian quantum Markov semigroups with a faithful normal invariant state. A unitary second-quantization representation for every $ s $-embedding shows that compactness at one positive time, compact resolvent, drift stability, and uniqueness of the normal invariant state are equivalent. These conditions imply immediate trace-class membership for all embeddings. We derive exact trace-norm and Hilbert--Schmidt norm formulas, an embedding-independent trace identity, and the generator spectrum with algebraic multiplicities. Examples show that compactness does not require the existence of a spectral gap and that a spectral gap need not imply compactness when the invariant state is not unique.
\end{abstract}

{\small
}

\section{Introduction} \label{section-introduction}

Gaussian quantum Markov semigroups (QMSs), also called quasi-free QMSs, describe bosonic dynamics with a quadratic Hamiltonian and linear noise operators. Their action on Weyl operators is determined by a finite-dimensional drift, a diffusion operator, and a displacement; see \eqref{eq-gaussian-action} and \eqref{eq-gaussian-multiplier}. In this paper, we study compactness of the semigroups induced on the Hilbert--Schmidt space, together with compactness of their generator resolvents. We also examine the relation between compactness and coercive spectral gaps.

Let $ ( \cT_t ) _{t \geq 0} $ be a Gaussian QMS on $ \cB ( \Gamma ( \C^d ) ) $, with $ d \geq 1 $, and let $ \rho $ be a faithful normal invariant state. For each $ s \in [ 0, 1 ] $, the prescription
\begin{equation}
    \label{eq-construction-induced-semigroup}
    T_t^{ ( s ) } \bigl ( \rho^{s/2} X \rho^{ ( 1 - s ) /2} \bigr ) = \rho^{s/2} \cT_t ( X ) \rho^{ ( 1 - s ) /2},
\end{equation}
for $ X \in \cB ( \Gamma ( \C^d ) ) $ and $ t \geq 0 $, extends to a strongly continuous contraction semigroup on $ \cB_2 ( \Gamma ( \C^d ) ) $; see \cite[Theorem~2.3]{carbone2000exponential}. We call $ ( T_t^{ ( s ) } ) _{t \geq 0} $ the \emph{induced semigroup} and denote its generator by $ L^{ ( s ) } $. The parameters $ s = 0 $, $ s = 1/2 $, and $ s = 1 $ give the GNS, KMS, and anti-GNS embeddings, respectively. The same construction can be formulated in the standard form of a $ \sigma $-finite von Neumann algebra. More general embeddings are obtained from inner products associated with operator monotone functions; see \cite[Section~2.3 and Proposition~3.5]{Wirth2026}.

The second-quantization approach to quasi-free quantum dynamics goes back to the pioneering work of Evans and Lewis. Their 1977 monograph constructs quasi-free completely positive maps and their contraction implementations in cyclic representations. In the Fock representation, they relate the implementation to second quantization \cite[Chapter~10]{EvansLewis1977}. This early connection between quasi-free maps and one-particle dynamics is the point of departure for the representation studied in the present paper.

Subsequent work developed second-quantization methods in several directions. Araki, Sinha, and Sunder characterized boundedness, compactness, and Schatten-class membership of second-quantized operators, and proved continuity on their natural Schatten domains \cite[Theorems~3.3 and~3.10]{ArakiSinhaSunder1995}. In classical Ornstein--Uhlenbeck theory, the representation for nonsymmetric semigroups obtained by Chojnowska-Michalik and Goldys in \cite{ChojnowskaMichalikGoldys1996} is recalled in \cite[Lemma~2.2]{ChojnowskaMichalikGoldys2002}. Van Neerven developed the corresponding spectral theory \cite{vanNeerven2005}. Spectral descriptions of particular quantum Ornstein--Uhlenbeck semigroups were obtained in \cite{cipriani2000spectral,ko2004construction}. More recently, Liu, Wan, and Wu used a finite-mode KMS second-quantization representation in their study of entropy contraction and hypercontractivity \cite[Proposition~2.2]{LiuWanWu2026}. Our purpose is to give a common representation for the full family of $ s $-embeddings and to use it to characterize compactness and its spectral consequences.

The explicit construction first assumes that $ \rho $ is Gaussian, with covariance matrix $ \mathbf{S} $. We prove that there is a unitary $ U_s : \Gamma ( \C^{2d} ) \to \cB_2 ( \Gamma ( \C^d ) ) $ such that
\begin{equation}
    \label{eq-intro-second-quantization}
    U_s^* T_t^{ ( s ) } U_s = \Gamma ( \me^{t \mathbf{Z}_s} ),
    \qquad
    \mathbf{Z}_s = \mathbf{S}_s^{1/2} \mathbf{Z} \mathbf{S}_s^{ - 1/2},
\end{equation}
where $ \mathbf{Z} $ is the real matrix representing the drift $ Z $ in \eqref{eq-gaussian-action} and
\begin{equation}
    \label{eq-definition-F-s}
    \mathbf{S}_s := ( \mathbf{S} + \mi \mathbf{J} ) \#_s ( \mathbf{S} - \mi \mathbf{J} )
\end{equation}
is the \emph{$ s $-weighted covariance matrix} associated with $ \rho $. Here $ \mathbf{J} $ is the matrix of the symplectic form in \eqref{eq-real-conventions}. For positive definite matrices $ A $ and $ B $, their $ s $-weighted geometric mean is
\begin{equation}
    \label{eq-geometric-mean}
    A \#_s B := A^{1/2} \bigl ( A^{ - 1/2} B A^{ - 1/2} \bigr ) ^s A^{1/2},
    \qquad 0 \leq s \leq 1.
\end{equation}
Faithfulness makes both matrices in \eqref{eq-definition-F-s} positive definite. Theorem~\ref{theorem-all-s-representation} establishes \eqref{eq-intro-second-quantization} for every embedding, and Corollary~\ref{corollary-generator-domain} identifies the closed generator, its domain and a finite-particle core. Neither drift stability nor uniqueness of the invariant state is required. If the faithful normal invariant state is not Gaussian, Proposition~\ref{proposition-state-independence} provides a faithful Gaussian reference state through the invariant-state decomposition.

Girotti and Poletti established the equivalence between stability of $ \mathbf{Z} $ and uniqueness of the normal invariant state in \cite[Theorems~11 and~13]{girotti2025gaussian}. Theorem~\ref{theorem-main} connects these conditions with compactness of the induced semigroups: they are equivalent to compact resolvent for one induced generator, compactness of one induced operator at a positive time, and immediate trace-class membership for every embedding. Compactness for a single embedding therefore implies compactness for all embeddings. Under these equivalent conditions, Corollary~\ref{corollary-traces} gives exact trace-norm and Hilbert--Schmidt norm formulas and the embedding-independent trace identity \eqref{eq-all-s-trace}. Corollary~\ref{corollary-spectrum} determines the full generator spectrum and its algebraic multiplicities, without assuming that the drift is diagonalizable.

We also revisit the relation between a positive spectral gap and compact resolvent, motivated by the one-mode result in \cite[Proposition~4.8]{fagnola_li_2025spectral}. We measure the coercive spectral gap, defined in \eqref{eq-definition-coercive-gap}, on the orthogonal complement of the full generator kernel. Under drift stability, Proposition~\ref{proposition-gap-formula} identifies the gap with the least eigenvalue of the dissipative part of the transformed drift. For a positive gap, Proposition~\ref{proposition-gap-resolvent} gives a direct finite-mode proof of compact resolvent through estimates on the particle subspaces, with an explicit approximation bound. Example~\ref{example-coupled-modes} shows that compact resolvent does not imply a positive gap, even when the invariant state is unique. Example~\ref{example-gap-infinite-kernel} shows why uniqueness cannot be omitted from the positive-gap implication: it has nonunique faithful normal invariant states and a positive gap off the full kernel, but noncompact generator resolvents.

The paper is organized as follows. Section~\ref{section-preliminaries} fixes the Gaussian framework. Section~\ref{section-second-quantization} constructs the second-quantization representation and identifies the induced generator, its domain and a core. Section~\ref{section-proof-main} proves the compactness criterion and derives its norm, trace and spectral consequences. Section~\ref{section-gaps-compactness} studies the relation between spectral gaps and compact resolvents, with a quantitative proof under uniqueness and examples showing the limits of this implication.

\section{Gaussian framework and induced semigroups} \label{section-preliminaries}

We recall the Gaussian framework needed for the second-quantization construction, following the notation and conventions of \cite{agredo2022decoherence,fagnola2025spectral}.

\subsection{Gaussian states and semigroups} \label{subsection-gaussian}

We use $ \N = \{ 0, 1, 2, \ldots \} $ and scalar products linear in the second variable. For a complex Hilbert space $ E $, let $ E^{ \odot n} $ denote the closed subspace of $ E^{ \otimes n} $ consisting of tensors invariant under permutations of the factors, with $ E^{ \odot 0} = \C $. The symmetric Fock space is
\begin{equation*}
    \Gamma ( E ) = \bigoplus_{n = 0}^\infty E^{ \odot n}.
\end{equation*}
We call $ E^{ \odot n} $ the \emph{$ n $-particle subspace}. A vector $ f \in \Gamma ( E ) $ is an orthogonal sum $ f = \bigoplus_{n = 0}^\infty f_n $, where $ f_n \in E^{ \odot n} $ and $ \norm{f}^2 = \sum_{n = 0}^\infty \norm{f_n}^2 < \infty $. Thus $ f_n $ is the $ n $-particle component of $ f $. We identify each $ E^{ \odot n} $ with the subspace of Fock space whose other particle components are zero. Vectors with only finitely many nonzero particle components are called \emph{finite-particle vectors}.

We use the symmetric tensor construction of \cite[Sections~17 and~19]{parthasarathy1992quantum}. For vectors or symmetric tensors, $ \odot $ denotes the orthogonal symmetrization of their tensor product over all underlying one-particle factors. The exponential vectors are
\begin{equation}
    \label{eq-exponential-vectors}
    e ( \xi ) = \bigoplus_{n = 0}^\infty \frac{\xi^{ \otimes n}}{\sqrt{n!}}, \qquad
    \langle e ( \xi ), e ( \eta ) \rangle = \me^{\langle \xi, \eta \rangle}, \qquad \xi, \eta \in E.
\end{equation}
The vector $ e ( 0 ) = 1 \oplus 0 \oplus \cdots $ is the unit vacuum. The linear span of the exponential vectors is dense in $ \Gamma ( E ) $. For finite-dimensional $ E $, the map $ \xi \mapsto e ( \xi ) $ is entire, as follows from the locally norm-convergent series in \eqref{eq-exponential-vectors}.

Let $ u_1, \ldots, u_d $ be the standard orthonormal basis of $ \C^d $. For $ n_1, \ldots, n_d \in \N $ with $ n = n_1 + \cdots + n_d \geq 1 $, define the normalized occupation-number vector by
\begin{equation*}
    \ket{n_1, \ldots, n_d} := \sqrt{\frac{n!}{n_1! \cdots n_d!}} \, u_1^{ \odot n_1} \odot \cdots \odot u_d^{ \odot n_d} \in ( \C^d ) ^{ \odot n}.
\end{equation*}
Factors with $ n_j = 0 $ are omitted, and the symmetrization is over all $ n $ one-particle factors. Set $ \ket{0, \ldots, 0} := e ( 0 ) $. Thus $ \ket{n_1, \ldots, n_d} $ has $ n_j $ particles in mode $ u_j $ and is regarded as a Fock vector with only its $ n $-particle component nonzero. These vectors form an orthonormal basis, and their finite linear span is the finite-particle space. Under the canonical mode decomposition $ \Gamma ( \C^d ) \simeq \Gamma ( \C ) ^{ \otimes d} $, they correspond to $ \ket{n_1} \otimes \cdots \otimes \ket{n_d} $. We use ket notation for this basis and $ e ( \xi ) $ for exponential vectors. For vectors $ u $ and $ v $, rank-one notation means $ \ketbra{u}{v} w = u \langle v, w \rangle $; in particular, $ \ketbra{n_1, \ldots, n_d}{n_1, \ldots, n_d} $ projects onto the corresponding number vector.

The annihilation and creation operators on mode $ j $ are the closures of the following actions on the finite-particle span:
\begin{equation*}
    \begin{aligned}
        a_j \ket{n_1, \ldots, n_d} & = \sqrt{n_j} \, \ket{n_1, \ldots, n_j - 1, \ldots, n_d},\\
        a_j^\dagger \ket{n_1, \ldots, n_d} & = \sqrt{n_j + 1} \, \ket{n_1, \ldots, n_j + 1, \ldots, n_d}.
    \end{aligned}
\end{equation*}
The first expression is zero when $ n_j = 0 $. These closed operators satisfy $ a_j^\dagger = a_j^* $ and have the finite-particle span as a common invariant core. The number operator $ N_j = a_j^\dagger a_j $ satisfies $ N_j \ket{n_1, \ldots, n_d} = n_j \ket{n_1, \ldots, n_d} $. In one mode we omit the mode subscript.

For $ z \in \C^d $, the Weyl operator is
\begin{equation*}
    W ( z ) = \exp{ \sum_{j = 1}^d ( z_j a_j^\dagger - \overline{z_j} a_j ) } .
\end{equation*}
Here the exponent denotes the skew-adjoint closure of the displayed action on the finite-particle span.
The Weyl relations are
\begin{equation}
    \label{eq-weyl-relations}
    W ( z ) W ( w ) = \me^{ - \mi \Im \langle z, w \rangle} W ( z + w ),
    \qquad W ( z ) ^* = W ( - z ) .
\end{equation}
The linear span of the Weyl operators is $ \sigma $-weakly dense in $ \cB ( \Gamma ( \C^d ) ) $.

We regard $ \C^d $ as a real Hilbert space with scalar product $ \Re \langle \cdot, \cdot \rangle $ and identify it isometrically with $ \R^{2d} $ through the grouped coordinates
\begin{equation}
    \label{eq-real-conventions}
    \mathbf{z} := \begin{bmatrix} \Re z \\ \Im z \end{bmatrix},
    \qquad \mathbf{J} := \begin{bmatrix} 0 & \1_d \\ - \1_d & 0 \end{bmatrix},
    \qquad \Im \langle z, w \rangle = \mathbf{z}^T \mathbf{J} \mathbf{w}.
\end{equation}
Thus $ \mathbf{y}^T \mathbf{z} = \Re \langle y, z \rangle $, and $ \mathbf{J} $ represents multiplication by $ - \mi $. Every real-linear operator $ A : \C^d \to \C^d $ has a unique decomposition
\begin{equation*}
    A z = A_1 z + A_2 \overline{z}, \qquad A_1, A_2 \in \C^{d \times d}.
\end{equation*}
Its real matrix representation is
\begin{equation}
    \label{eq-real-matrix-representation}
    \mathbf{A} := \begin{bmatrix}
        \Re A_1 + \Re A_2 & \Im A_2 - \Im A_1 \\
        \Im A_1 + \Im A_2 & \Re A_1 - \Re A_2
    \end{bmatrix}.
\end{equation}
It sends the grouped real coordinates of $ z $ to those of $ A z $. Conversely, every real $ 2d \times 2d $ matrix represents a unique real-linear operator in these coordinates. The adjoint $ A^\sharp $ is defined by
\begin{equation*}
    \Re \langle y, A z \rangle = \Re \langle A^\sharp y, z \rangle, \qquad y, z \in \C^d,
\end{equation*}
and is represented by $ \mathbf{A}^T $. The representation preserves compositions, powers and exponentials; in particular, $ \me^{t \mathbf{A}} $ represents $ \me^{tA} $. When a real matrix acts on $ \C^{2d} $ below, we use its complex-linear extension. For vectors, $ \abs{ \cdot } $ denotes the Euclidean norm.

Write $ \cB_1 ( \Gamma ( \C^d ) ) $ and $ \cB_2 ( \Gamma ( \C^d ) ) $ for the trace-class and Hilbert--Schmidt operators on $ \Gamma ( \C^d ) $. We identify a normal state with its positive density operator $ \rho \in \cB_1 ( \Gamma ( \C^d ) ) $, with $ \Tr \rho = 1 $. It is faithful if $ \ker \rho = \{ 0 \} $. Its characteristic function is $ \chi_\rho ( z ) = \Tr ( \rho W ( z ) ) $. The state is Gaussian if there exist $ \omega \in \C^d $ and a symmetric real-linear covariance operator $ S $ such that
\begin{equation} \label{eq-gaussian-state}
    \chi_\rho ( z )
    = \exp{ 2 \mi \Im \langle \omega, z \rangle - \frac{1}{2} \Re \langle z, S z \rangle } = \exp{ 2 \mi \boldsymbol{\omega}^T \mathbf{J} \mathbf{z} - \frac{1}{2} \mathbf{z}^T \mathbf{S} \mathbf{z} },
\end{equation}
where $ \boldsymbol{\omega} = ( \Re \omega, \Im \omega ) ^T $ and $ \mathbf{S} $ is the real matrix representing $ S $. The vector $ \omega $ and the matrix $ \mathbf{S} $ are called the mean vector and covariance matrix, respectively. The covariance condition is $ \mathbf{S} + \mi \mathbf{J} \geq 0 $, and the Gaussian state is faithful precisely when $ \mathbf{S} + \mi \mathbf{J} > 0 $; see \cite[Section~2]{fagnola2025spectral}. Its characteristic function never vanishes. We allow a nonzero mean throughout.

On one mode, the faithful thermal density with inverse temperature $ \beta > 0 $ is
\begin{equation}
    \label{eq-thermal-density}
    \rho_{\beta} = ( 1 - \me^{ - \beta} ) \me^{ - \beta N}, \qquad N = a^\dagger a.
\end{equation}
This is the thermal-state formula in \cite[Section~4, Eq.~(4.1)]{parthasarathy2021pedagogical}.

A QMS on $ \cB ( \Gamma ( \C^d ) ) $ is a semigroup of normal unital completely positive maps, with $ \cT_0 = \1 $ and $ t \mapsto \cT_t ( X ) $ continuous in the $ \sigma $-weak topology for every $ X $. A density $ \rho $ is invariant if $ \Tr ( \rho \cT_t ( X ) ) = \Tr ( \rho X ) $ for every $ X $ and $ t \geq 0 $. A Gaussian QMS acts on Weyl operators according to
\begin{equation}
    \label{eq-gaussian-action}
    \cT_t ( W ( z ) ) = \phi_t ( z ) W ( \me^{tZ} z ),
\end{equation}
where
\begin{equation}
    \label{eq-gaussian-multiplier}
    \phi_t ( z ) = \exp{ - \frac{1}{2} \int_0^t \Re \langle \me^{rZ} z, C \me^{rZ} z \rangle \, \md r + \mi \int_0^t \Re \langle \zeta, \me^{rZ} z \rangle \, \md r } .
\end{equation}
Here $ Z $ and $ C $ are real-linear operators, $ C = C^\sharp \geq 0 $, and $ \zeta \in \C^d $. They are called the \emph{drift}, \emph{diffusion}, and \emph{displacement} parameters, respectively. Their real matrix representations are $ \mathbf{Z} $ and $ \mathbf{C} $. The drift matrix $ \mathbf{Z} $ is called \emph{stable} if every eigenvalue has strictly negative real part. Since the drift space is finite dimensional, stability is equivalent to $ \norm{\me^{t \mathbf{Z}}} \to 0 $ as $ t \to \infty $.

For the induced semigroups, compactness and trace-class membership refer to operators on the Hilbert--Schmidt space
\begin{equation} \label{eq-definition-mathcal-H}
    \cH := \cB_2 ( \Gamma ( \C^d ) ),
\end{equation}
whose scalar product is $ \langle Y_1, Y_2 \rangle_2 = \Tr ( Y_1^* Y_2 ) $. We write $ \cS_1 ( \cH ) $ and $ \cS_2 ( \cH ) $ for the trace-class and Hilbert--Schmidt operators on $ \cH $, with norms $ \norm{ \cdot }_{\cS_1} $ and $ \norm{ \cdot }_{\cS_2} $, respectively. These classes concern operators on $ \cH $, whereas $ \cB_1 ( \Gamma ( \C^d ) ) $ and $ \cB_2 ( \Gamma ( \C^d ) ) $ concern operators on the original Fock space. For a bounded operator $ B $, its modulus is $ \abs{B} := ( B^* B ) ^{1/2} $. We write $ \sigma ( A ) $ for the spectrum of a closed operator $ A $. The symbol $ \1_E $ denotes the identity map on a space $ E $, with the subscript omitted when clear; $ \1_m $ denotes the identity on $ \C^m $. The operator trace on $ \cH $ is denoted by $ \tr_{\cH} $, and an unsubscripted operator norm refers to the bounded-operator norm on the space under consideration. We use the standard ideal properties and Schatten H\"{o}lder inequalities; see \cite[Theorems~2.7(a) and~2.8]{Simon2005}.

\section{Second quantization for all \texorpdfstring{$ s $}{s}-embeddings}\label{section-second-quantization}

For a faithful Gaussian invariant state, we construct the second-quantization representation by identifying the embedded Weyl operators, after normalization by the characteristic function, with exponential vectors. The weighted covariance matrix determines their scalar products. We then identify the closed generator and its domain, before extending the representation to arbitrary faithful normal invariant states.

\subsection{Weighted covariance matrices}\label{subsection-metric}

For a faithful Gaussian state with covariance $ \mathbf{S} $, the $ s $-weighted covariance matrix $ \mathbf{S}_s $ in \eqref{eq-definition-F-s} is positive definite. We compute it in Williamson coordinates to prepare the scalar-product formula in Subsection~\ref{subsection-gram}. Invariance is not required for these calculations.

We use the Williamson decomposition with inverse temperatures $ \beta_1, \ldots, \beta_d > 0 $ and a real symplectic matrix $ \mathbf{G} $. We first collect the diagonal matrices needed below. For $ 0 \leq s \leq 1 $ and $ 1 \leq j \leq d $, set
\begin{equation}
    \label{eq-metric-coefficients}
    c_{s, j} = \frac{\cosh ( ( 1 - 2 s ) \beta_j/2 ) }{\sinh ( \beta_j/2 ) }, \qquad
    d_{s, j} = \frac{\sinh ( ( 1 - 2 s ) \beta_j/2 ) }{\sinh ( \beta_j/2 ) }.
\end{equation}
The corresponding diagonal matrices are
\begin{align}
    \label{eq-thermal-diagonal-matrices}
    D_{\coth} &= \diag_{1 \leq j \leq d} ( \coth ( \beta_j/2 ) ), \qquad
    D_{\csch} = \diag_{1 \leq j \leq d} ( \csch ( \beta_j/2 ) ), \\
    \label{eq-weighted-diagonal-matrices}
    D_{c_s} &= \diag ( c_{s, 1}, \ldots, c_{s, d} ), \qquad
    D_{d_s} = \diag ( d_{s, 1}, \ldots, d_{s, d} ).
\end{align}
The Williamson decomposition is
\begin{equation}
    \label{eq-williamson}
    \mathbf{S} = \mathbf{G}^T \mathbf{D}_{\coth} \mathbf{G}, \qquad
    \mathbf{G}^T \mathbf{J} \mathbf{G} = \mathbf{J}, \qquad
    \mathbf{D}_{\coth} = \begin{bmatrix}D_{\coth} & 0\\0 & D_{\coth} \end{bmatrix}.
\end{equation}
The corresponding density has the thermal product representation
\begin{equation}
    \label{eq-thermal-product}
    \rho = W ( \omega ) B ( \mathbf{G}^{ - 1} ) \rho_{\boldsymbol{\beta}} B ( \mathbf{G}^{ - 1} ) ^* W ( \omega ) ^*, \quad \rho_{\boldsymbol{\beta}} := \bigotimes_{j = 1}^d \rho_{\beta_j}.
\end{equation}
Each factor is the one-mode density in \eqref{eq-thermal-density}. The unitary Bogoliubov implementer satisfies
\begin{equation*}
    B ( \mathbf{G}^{ - 1} ) ^* W ( z ) B ( \mathbf{G}^{ - 1} ) = W ( Gz ),
\end{equation*}
where $ G $ is the real-linear map represented by $ \mathbf{G} $. We refer to \cite[Sections~3 and~4]{parthasarathy2021pedagogical} for these facts.

\begin{lemma}\label{lemma-explicit-metric}
For the diagonal matrices in \eqref{eq-thermal-diagonal-matrices} and \eqref{eq-weighted-diagonal-matrices}, and the symplectic matrix $ \mathbf{G} $ in \eqref{eq-williamson}, we have
\begin{equation}
    \label{eq-williamson-F-s}
    \mathbf{S}_s = \mathbf{G}^T
    \begin{bmatrix}D_{c_s} & \mi D_{d_s}\\-\mi D_{d_s} & D_{c_s} \end{bmatrix}
    \mathbf{G}.
\end{equation}
In particular,
\begin{equation}
    \label{eq-special-metrics}
    \mathbf{S}_0 = \mathbf{S} + \mi \mathbf{J}, \qquad \mathbf{S}_1 = \mathbf{S} - \mi \mathbf{J}, \qquad
    \mathbf{S}_{1/2} = \mathbf{G}^T \begin{bmatrix}D_{\csch} & 0\\0 & D_{\csch} \end{bmatrix}\mathbf{G}.
\end{equation}
Moreover, $ \mathbf{S}_{1 - s} = \overline{\mathbf{S}_s} $.
\end{lemma}
\begin{proof}
For positive definite matrices $ A $ and $ B $ of the same size and any invertible matrix $ R $, the weighted geometric mean satisfies
\begin{equation*}
    ( R^* A R ) \#_s ( R^* B R ) = R^* ( A \#_s B ) R.
\end{equation*}
Indeed, write $ A \#_s B = A ( A^{ - 1} B ) ^s $, where the power is defined by similarity with the positive matrix $ A^{ - 1/2} B A^{ - 1/2} $. The identity follows from
$ ( R^* A R ) ^{ - 1} ( R^* B R ) = R^{ - 1} ( A^{ - 1} B ) R $ and similarity invariance of the functional calculus.

Applying this identity to \eqref{eq-williamson}, it remains to compute
$ ( \mathbf{D}_{\coth} + \mi \mathbf{J} ) \#_s ( \mathbf{D}_{\coth} - \mi \mathbf{J} ) $. The two matrices commute. On the two-dimensional coordinate subspace corresponding to mode $ j $, the eigenvalues of $ \mi \mathbf{J} $ are $ \pm 1 $, and
\begin{equation*}
    \coth ( \beta_j/2 ) \pm 1 = \frac{\me^{ \pm \beta_j/2}}{\sinh ( \beta_j/2 ) }.
\end{equation*}
The eigenvalues of their weighted geometric mean are thus
$ \me^{ \pm ( 1 - 2 s ) \beta_j/2}/\sinh ( \beta_j/2 ) = c_{s, j} \pm d_{s, j} $, where $ c_{s, j} $ and $ d_{s, j} $ are the coefficients in \eqref{eq-metric-coefficients}. This proves \eqref{eq-williamson-F-s}, the special cases in \eqref{eq-special-metrics}, and the identity $ \mathbf{S}_{1 - s} = \overline{\mathbf{S}_s} $. The same expressions also verify positivity directly.
\end{proof}

\subsection{The weighted Weyl pairing}\label{subsection-gram}

We compute the scalar products of the embedded Weyl operators from the one-mode pairing in \cite[Lemma~5.1]{fagnola_li_2025spectral}. We first record the formula in the normalization used here; the KMS case is also treated in \cite[Appendix~C, Proposition~35]{fagnola2025spectral}.

\begin{lemma}\label{lemma-thermal-trace}
Let $ \rho_{\beta} $ be the one-mode thermal density in \eqref{eq-thermal-density}. For $ z, w \in \C $ and $ 0 \leq s \leq 1 $,
\begin{equation}
    \label{eq-thermal-trace}
    \Tr \bigl ( \rho_{\beta}^{1 - s} W ( - z ) \rho_{\beta}^s W ( w ) \bigr ) = \exp{
        - \frac{1}{2} \coth ( \beta/2 ) ( \abs{z}^2 + \abs{w}^2 )
        + \frac{\me^{ - s \beta} \overline{z} w + \me^{ - ( 1 - s ) \beta} z \overline{w}}{1 - \me^{ - \beta}}
    }.
\end{equation}
\end{lemma}
\begin{proof}
Apply \cite[Lemma~5.1, Eq.~(5.3)]{fagnola_li_2025spectral} with inverse temperature $ 2 \beta $, embedding parameter $ 1 - s $, and Weyl arguments $ - z $ and $ w $. The identity
\begin{equation*}
    \rho_{\beta} = \sqrt{\tanh ( \beta/2 )} \rho_{2 \beta}^{1/2}
\end{equation*}
cancels the prefactor $ \sqrt{\coth ( \beta/2 )} $ in that formula and gives \eqref{eq-thermal-trace}. The cited identity holds for every $ s \in [ 0, 1 ] $, including the endpoints.
\end{proof}

\begin{proposition}\label{proposition-all-s-gram}
Let $ \rho $ be a faithful Gaussian state, not necessarily invariant. For every $ s \in [ 0, 1 ] $ and $ z, w \in \C^d $,
\begin{equation}
    \label{eq-all-s-gram}
    \left\langle \rho^{s/2} W ( z ) \rho^{ ( 1 - s ) /2},
    \rho^{s/2} W ( w ) \rho^{ ( 1 - s ) /2} \right\rangle_2 = \overline{\chi_\rho ( z )} \chi_\rho ( w )
    \exp{ \mathbf{z}^T \mathbf{S}_s \mathbf{w} }.
\end{equation}
\end{proposition}
\begin{proof}
Cyclicity of the trace gives
\begin{equation*}
    \left\langle \rho^{s/2} W ( z ) \rho^{ ( 1 - s ) /2},
    \rho^{s/2} W ( w ) \rho^{ ( 1 - s ) /2} \right\rangle_2  = \Tr \bigl ( \rho^{1 - s} W ( - z ) \rho^s W ( w ) \bigr ).
\end{equation*}
Schatten H\"{o}lder's inequality justifies the trace products and cyclic permutations, including $ s = 0 $ and $ s = 1 $. For one thermal mode with density $ \rho_{\beta} $, divide by $ \overline{\chi_{\rho_{\beta}} ( z )} \chi_{\rho_{\beta}} ( w ) $. The two diagonal terms in Lemma~\ref{lemma-thermal-trace} cancel, leaving the exponent
\begin{equation*}
    \frac{\me^{ - s \beta} \overline{z} w + \me^{ - ( 1 - s ) \beta} z \overline{w}}{1 - \me^{ - \beta}}
    = c_{s, 1} \Re \langle z, w \rangle + \mi d_{s, 1} \Im \langle z, w \rangle,
\end{equation*}
where the coefficients in \eqref{eq-metric-coefficients} are evaluated at $ \beta_1 = \beta $. The $ s $-weighted covariance matrix is $ c_{s, 1} \1_2 + \mi d_{s, 1} \mathbf{J} $. Taking tensor products gives the block matrix in \eqref{eq-williamson-F-s} for the thermal product state $ \rho_{\boldsymbol{\beta}} $ in \eqref{eq-thermal-product}.

For a general Gaussian state, use the thermal product representation \eqref{eq-thermal-product}. The Weyl relations and the Bogoliubov conjugation give
\begin{equation*}
    B ( \mathbf{G}^{ - 1} ) ^* W ( \omega ) ^* W ( z ) W ( \omega ) B ( \mathbf{G}^{ - 1} )
    = \me^{2 \mi \Im \langle \omega, z \rangle} W ( Gz ).
\end{equation*}
Hence
\begin{equation*}
    \chi_\rho ( z ) = \me^{2 \mi \Im \langle \omega, z \rangle} \chi_{\rho_{\boldsymbol{\beta}}} ( Gz ).
\end{equation*}
Conjugating the weighted trace by $ W ( \omega ) B ( \mathbf{G}^{ - 1} ) $, the Weyl factors $ W ( - z ) $ and $ W ( w ) $ contribute the phases $ \me^{ - 2 \mi \Im \langle \omega, z \rangle} $ and $ \me^{2 \mi \Im \langle \omega, w \rangle} $, respectively. Their product is $ \me^{2 \mi \Im \langle \omega, w - z \rangle} $, which is also the phase in
\begin{equation*}
    \overline{\chi_\rho ( z )} \chi_\rho ( w )
    = \me^{2 \mi \Im \langle \omega, w - z \rangle}
    \overline{\chi_{\rho_{\boldsymbol{\beta}}} ( Gz )} \chi_{\rho_{\boldsymbol{\beta}}} ( Gw ).
\end{equation*}
Thus unitary invariance of the trace and cancellation of this common factor give
\begin{equation*}
    \frac{\Tr \bigl ( \rho^{1 - s} W ( - z ) \rho^s W ( w ) \bigr )}{\overline{\chi_\rho ( z )} \chi_\rho ( w )}
    =
    \frac{\Tr \bigl ( \rho_{\boldsymbol{\beta}}^{1 - s} W ( - Gz ) \rho_{\boldsymbol{\beta}}^s W ( Gw ) \bigr )}{\overline{\chi_{\rho_{\boldsymbol{\beta}}} ( Gz )} \chi_{\rho_{\boldsymbol{\beta}}} ( Gw )}.
\end{equation*}
The thermal calculation applies to the right-hand side. Lemma~\ref{lemma-explicit-metric} identifies the resulting exponent as $ \mathbf{z}^T \mathbf{S}_s \mathbf{w} $, proving \eqref{eq-all-s-gram}.
\end{proof}

\subsection{Unitary representation and generator}\label{subsection-unitary}

For a matrix $ B $ on $ \C^{2d} $, define $ B^{ \odot n} := \left. B^{ \otimes n} \right|_{( \C^{2d} ) ^{ \odot n}} $ for $ n \geq 1 $, and $ B^{ \odot 0} := \1_{\C} $. The restriction is well defined because $ B^{ \otimes n} $ commutes with permutations of the tensor factors and hence preserves the symmetric subspace. For a contraction $ B $, its symmetric second quantization is
\begin{equation}
    \label{eq-definition-second-quantization}
    \Gamma ( B ) = \bigoplus_{n = 0}^\infty B^{ \odot n}, \qquad
    \Gamma ( B ) e ( \xi ) = e ( B \xi ),
\end{equation}
Here $ e ( \xi ) $ is the exponential vector in \eqref{eq-exponential-vectors}. The auxiliary Fock space $ \Gamma ( \C^{2d} ) $ has a $ 2d $-dimensional one-particle space. The following theorem identifies it with the Hilbert--Schmidt space $ \cH $.

\begin{theorem}[Second quantization for every embedding]\label{theorem-all-s-representation}
Suppose that $ \rho $ is a faithful Gaussian invariant state. For every $ s \in [ 0, 1 ] $ there is a unitary $ U_s : \Gamma ( \C^{2d} ) \to \cH $ determined by
\begin{equation}
    \label{eq-all-s-unitary}
    U_s \, e ( \mathbf{S}_s^{1/2} \mathbf{z} ) = \chi_\rho ( z ) ^{ - 1} \rho^{s/2} W ( z ) \rho^{ ( 1 - s ) /2}, \qquad z \in \C^d.
\end{equation}
The transformed drift $ \mathbf{Z}_s $ in \eqref{eq-intro-second-quantization} satisfies
\begin{equation}
    \label{eq-definition-A-s}
    \me^{t \mathbf{Z}_s} = \mathbf{S}_s^{1/2} \me^{t \mathbf{Z}} \mathbf{S}_s^{ - 1/2}.
\end{equation}
It generates a contraction semigroup $ ( \me^{t \mathbf{Z}_s} ) _{t \geq 0} $ on $ \C^{2d} $, and
\begin{equation}
    \label{eq-all-s-second-quantization}
    U_s^*T_t^{ ( s ) } U_s = \Gamma ( \me^{t \mathbf{Z}_s} ), \qquad t \geq 0.
\end{equation}
\end{theorem}
\begin{proof}
Proposition~\ref{proposition-all-s-gram} shows that the prescription in \eqref{eq-all-s-unitary} preserves scalar products. We first verify totality. If $ f \in \Gamma ( \C^{2d} ) $ is orthogonal to every $ e ( \mathbf{S}_s^{1/2} x ) $, $ x \in \R^{2d} $, then the entire function
$ x \mapsto \langle f, e ( \mathbf{S}_s^{1/2} x ) \rangle $
vanishes on $ \R^{2d} $. Successive applications of the one-variable identity theorem show that it vanishes on $ \C^{2d} $. Since $ \mathbf{S}_s^{1/2} $ is invertible and exponential vectors are total, $ f = 0 $.

If $ Y \in \cH $ is orthogonal to every embedded Weyl operator, then
\begin{equation*}
    \Tr \bigl ( \rho^{ ( 1 - s ) /2} Y^* \rho^{s/2} W ( z ) \bigr ) = 0,
    \qquad z \in \C^d.
\end{equation*}
The sandwiched operator is trace class by Schatten H\"{o}lder's inequality, including at $ s = 0 $ and $ s = 1 $. The $ \sigma $-weak density of the Weyl span therefore gives
$ \rho^{ ( 1 - s ) /2} Y^* \rho^{s/2} = 0 $.
Faithfulness of $ \rho $ implies $ Y = 0 $. Thus \eqref{eq-all-s-unitary} extends uniquely to a unitary $ U_s $.

Set $ V_t = U_s^* T_t^{ ( s ) } U_s $. It is contractive because $ T_t^{ ( s ) } $ is a contraction and $ U_s $ is unitary. Equation~\eqref{eq-all-s-unitary} gives
\begin{equation*}
    U_s^* \bigl ( \rho^{s/2} W ( z ) \rho^{ ( 1 - s ) /2} \bigr )
    = \chi_\rho ( z ) e ( \mathbf{S}_s^{1/2} \mathbf{z} ), \qquad z \in \C^d.
\end{equation*}
Invariance gives $ \chi_\rho ( z ) = \phi_t ( z ) \chi_\rho ( \me^{tZ} z ) $. Thus \eqref{eq-construction-induced-semigroup} and \eqref{eq-gaussian-action} yield
\begin{equation*}
    V_t \, e ( \mathbf{S}_s^{1/2} \mathbf{z} )
    = \frac{\phi_t ( z )}{\chi_\rho ( z )}
    U_s^* \bigl ( \rho^{s/2} W ( \me^{tZ} z ) \rho^{ ( 1 - s ) /2} \bigr ) = \frac{\phi_t ( z ) \chi_\rho ( \me^{tZ} z )}{\chi_\rho ( z )}
    e ( \mathbf{S}_s^{1/2} \me^{t \mathbf{Z}} \mathbf{z} ) = e ( \mathbf{S}_s^{1/2} \me^{t \mathbf{Z}} \mathbf{z} ).
\end{equation*}
Every $ x \in \R^{2d} $ has the form $ \mathbf{z} $ for a unique $ z \in \C^d $, so
$ V_t \, e ( \mathbf{S}_s^{1/2} x ) = e ( \mathbf{S}_s^{1/2} \me^{t \mathbf{Z}} x ) $.
Both sides are entire vector-valued functions of $ x $. Applying the same uniqueness argument after taking scalar products extends the identity to every $ x \in \C^{2d} $. Since $ \mathbf{S}_s^{1/2} $ is invertible, this gives
\begin{equation*}
    V_t \, e ( \xi ) = e ( \me^{t \mathbf{Z}_s} \xi ), \qquad \xi \in \C^{2d}.
\end{equation*}
Contractivity of $ V_t $ and the identity $ \norm{e ( \xi )} = \exp{ \abs{\xi}^2/2 } $ imply $ \abs{\me^{t \mathbf{Z}_s} \xi} \leq \abs{\xi} $ for every $ \xi \in \C^{2d} $. Thus $ \norm{\me^{t \mathbf{Z}_s}} \leq 1 $, and \cite[Theorem~3.3(a)]{ArakiSinhaSunder1995} ensures that $ \Gamma ( \me^{t \mathbf{Z}_s} ) $ is bounded. To obtain its contraction bound, \cite[Lemma~3.1]{ArakiSinhaSunder1995} gives
\begin{equation*}
    \norm{( \me^{t \mathbf{Z}_s} ) ^{ \odot n}} = \norm{\me^{t \mathbf{Z}_s}}^n \leq 1, \qquad n \geq 1,
\end{equation*}
and the vacuum block is the identity on $ \C $. Consequently, for every $ f = \bigoplus_{n = 0}^\infty f_n \in \Gamma ( \C^{2d} ) $, orthogonality of the particle subspaces gives
\begin{equation*}
    \norm{\Gamma ( \me^{t \mathbf{Z}_s} ) f}^2
    = \sum_{n = 0}^\infty \norm{( \me^{t \mathbf{Z}_s} ) ^{ \odot n} f_n}^2
    \leq \sum_{n = 0}^\infty \norm{f_n}^2 = \norm{f}^2.
\end{equation*}
Hence $ \Gamma ( \me^{t \mathbf{Z}_s} ) $ is a contraction; its operator norm is exactly one because it fixes the unit vacuum vector. It agrees with $ V_t $ on the total set of exponential vectors. Since both operators are bounded, this proves \eqref{eq-all-s-second-quantization}.
\end{proof}

Following \cite[Eq.~(20.11) and Proposition~20.7]{parthasarathy1992quantum}, we denote by $ \lambda ( H ) $ the \emph{differential second quantization} of a self-adjoint one-particle matrix $ H $. It is the self-adjoint generator determined by the unitary-group identity
\begin{equation*}
    \Gamma \bigl ( \exp{ - \mi t H } \bigr ) = \exp{ - \mi t \lambda ( H ) }, \qquad t \in \R.
\end{equation*}
On finite-particle vectors, $ \lambda ( H ) $ applies $ H $ to each tensor factor and sums the results, and it vanishes on the vacuum. Parthasarathy extends this notation to bounded one-particle operators by complex linearity \cite[Eq.~(20.30) and Proposition~20.14(ii)]{parthasarathy1992quantum}. For an arbitrary matrix $ A $ on a finite-dimensional one-particle space, we take $ \lambda ( A ) $ to be the closure of that finite-particle action. The action is closable because its adjoint contains the densely defined finite-particle action of $ A^* $.

Suppose that $ \exp{ t A } $ is contractive for $ t \geq 0 $. Each symmetric tensor power $ ( \exp{ t A } ) ^{ \odot n} $ is then contractive, and the tensor-product rule gives the semigroup law for $ \Gamma ( \exp{ t A } ) $. Continuity on finite-particle vectors, their density in Fock space, and the common contraction bound prove strong continuity. Differentiation on each particle subspace gives the finite-particle action defining $ \lambda ( A ) $. The direct-sum domain and graph-norm argument in the proof of Corollary~\ref{corollary-generator-domain} below identify its closure with the full generator. That argument uses only finite dimensionality and contractivity, and therefore applies to any such $ A $. Consequently,
\begin{equation*}
    \Gamma \bigl ( \exp{ t A } \bigr ) = \exp{ t \lambda ( A ) }, \qquad t \geq 0.
\end{equation*}
Here the exponential on the right denotes the strongly continuous semigroup generated by the generally unbounded operator $ \lambda ( A ) $. Taking $ A = - \mi H $ recovers the unitary identity, since $ \lambda ( - \mi H ) = - \mi \lambda ( H ) $. The same construction appears, in different notation, in van Neerven's treatment of second-quantized semigroups \cite[Section~4]{vanNeerven2005}.

Theorem~\ref{theorem-all-s-representation} places $ \mathbf{Z}_s $ in this contraction-semigroup framework. The unitary $ U_s $ changes the representation of the induced semigroup; the dynamics need not be unitary, since $ \mathbf{Z}_s $ need not be skew-adjoint. The following result applies the standard generator construction for second quantization to $ A = \mathbf{Z}_s $. The finite-particle tensor-sum formula is given in \cite[Proposition~20.14(ii), Eq.~(20.33)]{parthasarathy1992quantum}, while the contraction-semigroup generator and finite-particle core are established in \cite[Section~3.2, Lemma~3.1]{GeorgescuGerardMoller2004}. We give a direct proof to make the domain formula and the graph-norm approximation by finite-particle vectors explicit.

\begin{corollary}\label{corollary-generator-domain}
Under the assumptions of Theorem~\ref{theorem-all-s-representation},
\begin{equation}
    \label{eq-second-quantized-generator}
    U_s^*L^{ ( s ) } U_s = \lambda ( \mathbf{Z}_s ) = \bigoplus_{n = 0}^\infty \mathbf{Z}_{s, n}, \qquad
    \mathbf{Z}_{s, n} := \left.\sum_{j = 1}^n\1^{ \otimes ( j - 1 ) } \otimes \mathbf{Z}_s \otimes \1^{ \otimes ( n - j ) }
    \right|_{( \C^{2d} ) ^{ \odot n}},
\end{equation}
where $ \mathbf{Z}_{s, 0} = 0 $. The domain of $ U_s^* L^{ ( s ) } U_s $ is
\begin{equation}
    \label{eq-second-quantized-domain}
    \left\{ f = \bigoplus_{n = 0}^\infty f_n \in \Gamma ( \C^{2d} ) : \sum_{n = 0}^\infty \norm{\mathbf{Z}_{s, n} f_n}^2 < \infty \right\}.
\end{equation}
Here $ f_n $ is the orthogonal projection of $ f $ onto the $ n $-particle subspace $ ( \C^{2d} ) ^{ \odot n} $, as in Subsection~\ref{subsection-gaussian}. Vectors with only finitely many nonzero particle components form a core for $ U_s^* L^{ ( s ) } U_s $, and their images under $ U_s $ form a core for $ L^{ ( s ) } $.
\end{corollary}

\begin{proof}
Fix $ s \in [ 0, 1 ] $. For $ n \geq 1 $, define operators on the full tensor product $ ( \C^{2d} ) ^{ \otimes n} $ by
\begin{equation*}
    B_j := \1^{ \otimes ( j - 1 ) } \otimes \mathbf{Z}_s \otimes \1^{ \otimes ( n - j ) }, \qquad j = 1, \ldots, n.
\end{equation*}
For $ j \neq k $, both $ B_j B_k $ and $ B_k B_j $ apply $ \mathbf{Z}_s $ to the $ j $-th and $ k $-th factors and leave the other factors unchanged, so $ B_j B_k = B_k B_j $. For every integer $ m \geq 0 $, the tensor-product rule gives
\begin{equation*}
    B_j^m = \1^{ \otimes ( j - 1 ) } \otimes \mathbf{Z}_s^m \otimes \1^{ \otimes ( n - j ) },
\end{equation*}
where $ \mathbf{Z}_s^0 = \1 $. The norm-convergent exponential series therefore gives
\begin{equation*}
    \exp{ t B_j } = \sum_{m = 0}^\infty \frac{t^m B_j^m}{m!} = \1^{ \otimes ( j - 1 ) } \otimes
    \left( \sum_{m = 0}^\infty \frac{t^m}{m!} \mathbf{Z}_s^m \right)
    \otimes \1^{ \otimes ( n - j ) } = \1^{ \otimes ( j - 1 ) } \otimes \exp{ t \mathbf{Z}_s } \otimes \1^{ \otimes ( n - j ) }.
\end{equation*}
Thus $ \exp{ t B_j } $ applies $ \exp{ t \mathbf{Z}_s } $ to the $ j $-th tensor factor and leaves the others unchanged. For $ x_1, \ldots, x_n \in \C^{2d} $, it follows that
\begin{equation*}
    \begin{aligned}
        \left( \prod_{j = 1}^n \exp{ t B_j } \right) ( x_1 \otimes \cdots \otimes x_n )
        & = ( \exp{ t \mathbf{Z}_s } x_1 ) \otimes \cdots \otimes ( \exp{ t \mathbf{Z}_s } x_n )\\
        & = ( \exp{ t \mathbf{Z}_s } ) ^{ \otimes n} ( x_1 \otimes \cdots \otimes x_n ).
    \end{aligned}
\end{equation*}
Since elementary tensors span the full tensor product, this equality holds as an operator identity. Pairwise commutativity of the $ B_j $ then gives
\begin{equation*}
    \exp{ t \sum_{j = 1}^n B_j }
    = \prod_{j = 1}^n \exp{ t B_j }
    = ( \exp{ t \mathbf{Z}_s } ) ^{ \otimes n}, \qquad t \geq 0.
\end{equation*}
Every permutation of the tensor factors merely relabels the summands $ B_j $. Their sum therefore commutes with every such permutation and preserves the symmetric subspace. By the definition in \eqref{eq-second-quantized-generator}, the restriction of this sum to $ ( \C^{2d} ) ^{ \odot n} $ is $ \mathbf{Z}_{s, n} $. Since each successive application of the sum remains in this subspace, induction on $ m $ gives
\begin{equation*}
    \left( \sum_{j = 1}^n B_j \right)^m f_n = \mathbf{Z}_{s, n}^m f_n,
    \qquad f_n \in ( \C^{2d} ) ^{ \odot n}, \quad m \geq 0.
\end{equation*}
For every such $ f_n $, the norm-convergent exponential series consequently gives
\begin{equation*}
    \exp{ t \sum_{j = 1}^n B_j } f_n = \sum_{m = 0}^\infty \frac{t^m}{m!} \left( \sum_{j = 1}^n B_j \right)^m f_n = \sum_{m = 0}^\infty \frac{t^m}{m!} \mathbf{Z}_{s, n}^m f_n = \exp{ t \mathbf{Z}_{s, n} } f_n.
\end{equation*}
By \eqref{eq-definition-second-quantization}, the symmetric tensor power is the restriction of the full tensor power to $ ( \C^{2d} ) ^{ \odot n} $. Combining the preceding identities therefore gives
\begin{equation*}
    ( \exp{ t \mathbf{Z}_s } ) ^{ \odot n} f_n = ( \exp{ t \mathbf{Z}_s } ) ^{ \otimes n} f_n = \exp{ t \sum_{j = 1}^n B_j } f_n = \exp{ t \mathbf{Z}_{s, n} } f_n.
\end{equation*}
Since this holds for every $ f_n \in ( \C^{2d} ) ^{ \odot n} $, it proves the operator identity
\begin{equation*}
    ( \exp{ t \mathbf{Z}_s } ) ^{ \odot n} = \exp{ t \mathbf{Z}_{s, n} }, \qquad t \geq 0.
\end{equation*}
For $ n = 0 $, this identity reads $ 1 = \exp{ 0 } $. Now let $ f = \bigoplus_{n = 0}^\infty f_n \in \Gamma ( \C^{2d} ) $ and $ t > 0 $. Combining Theorem~\ref{theorem-all-s-representation} with these particle-subspace identities gives
\begin{equation*}
    U_s^* T_t^{ ( s ) } U_s f = \Gamma ( \me^{t \mathbf{Z}_s} ) f = \bigoplus_{n = 0}^\infty ( \me^{t \mathbf{Z}_s} ) ^{ \odot n} f_n = \bigoplus_{n = 0}^\infty \me^{t \mathbf{Z}_{s, n}} f_n.
\end{equation*}
Using $ U_s^* U_s f = f $, we subtract and divide by $ t $ componentwise to obtain
\begin{equation*}
    U_s^* \frac{T_t^{ ( s ) } U_s f - U_s f}{t} = \frac{U_s^* T_t^{ ( s ) } U_s f - f}{t} = \frac{1}{t} \left( \bigoplus_{n = 0}^\infty \me^{t \mathbf{Z}_{s, n}} f_n - \bigoplus_{n = 0}^\infty f_n \right)
    = \bigoplus_{n = 0}^\infty \frac{\me^{t \mathbf{Z}_{s, n}} - \1}{t} f_n.
\end{equation*}
The last direct sum is well defined in Fock space for every fixed $ t > 0 $. Indeed, contractivity on each particle subspace gives
\begin{equation*}
    \sum_{n = 0}^\infty \norm{\frac{\me^{t \mathbf{Z}_{s, n}} - \1}{t} f_n}^2 \leq \frac{4}{t^2} \sum_{n = 0}^\infty \norm{f_n}^2 < \infty.
\end{equation*}
By unitarity of $ U_s $, the vector $ U_s f $ belongs to $ \operatorname{Dom} L^{ ( s ) } $ precisely when this difference quotient converges in Fock space as $ t \downarrow 0 $. Its limit is then $ U_s^* L^{ ( s ) } U_s f $.

Suppose first that $ U_s f \in \operatorname{Dom} L^{ ( s ) } $. The orthogonal projection onto the $ n $-particle subspace is bounded and therefore commutes with this limit. On that finite-dimensional subspace, the derivative of $ \me^{t \mathbf{Z}_{s, n}} f_n $ at zero is $ \mathbf{Z}_{s, n} f_n $. Thus the $ n $-particle component of $ U_s^* L^{ ( s ) } U_s f $ is $ \mathbf{Z}_{s, n} f_n $. Orthogonality and unitarity give
\begin{equation*}
    \sum_{n = 0}^\infty \norm{\mathbf{Z}_{s, n} f_n}^2
    = \norm{L^{ ( s ) } U_s f}_2^2 < \infty.
\end{equation*}
This proves the necessity of \eqref{eq-second-quantized-domain}.

Conversely, suppose that $ f $ belongs to the set in \eqref{eq-second-quantized-domain}. For each $ n $ and $ t > 0 $,
\begin{equation*}
    \frac{\me^{t \mathbf{Z}_{s, n}} - \1}{t} f_n
    = \frac{1}{t} \int_0^t \me^{r \mathbf{Z}_{s, n}} \mathbf{Z}_{s, n} f_n \, \md r.
\end{equation*}
Since $ \me^{r \mathbf{Z}_{s, n}} = ( \me^{r \mathbf{Z}_s} ) ^{ \odot n} $ is contractive, the norm of this quotient is at most $ \norm{\mathbf{Z}_{s, n} f_n} $. For each fixed $ n $, its difference from $ \mathbf{Z}_{s, n} f_n $ tends to zero and has norm at most $ 2 \norm{\mathbf{Z}_{s, n} f_n} $. By the square-summability assumption, dominated convergence applies to the sum of these squared norms. Hence the full difference quotient converges to $ \bigoplus_{n = 0}^\infty \mathbf{Z}_{s, n} f_n $, proving that $ U_s f \in \operatorname{Dom} L^{ ( s ) } $ and
\begin{equation*}
    L^{ ( s ) } U_s f = U_s \left( \bigoplus_{n = 0}^\infty \mathbf{Z}_{s, n} f_n \right).
\end{equation*}

Finally, for $ f $ in the domain \eqref{eq-second-quantized-domain}, let $ f^{ ( r ) } $ be its truncation to the first $ r + 1 $ particle components. Then
\begin{equation*}
    \norm{U_s f - U_s f^{ ( r ) }}_2^2
    + \norm{L^{ ( s ) } U_s f - L^{ ( s ) } U_s f^{ ( r ) }}_2^2
    = \sum_{n > r} \left( \norm{f_n}^2 + \norm{\mathbf{Z}_{s, n} f_n}^2 \right) \longrightarrow 0.
\end{equation*}
Thus $ U_s f^{ ( r ) } $ converges to $ U_s f $ in the graph norm of $ L^{ ( s ) } $. This proves both core assertions and identifies the conjugated generator with the closure of the finite-particle action defining $ \lambda ( \mathbf{Z}_s ) $. Uniqueness of the semigroup generated by this closed operator now gives $ \Gamma ( \exp{ t \mathbf{Z}_s } ) = \exp{ t \lambda ( \mathbf{Z}_s ) } $ for $ t \geq 0 $. The same difference-quotient and core argument applies to any finite-dimensional matrix $ A $ generating a contraction semigroup and proves the general contraction-semigroup identity stated above.
\end{proof}

\subsection{Arbitrary faithful invariant states}\label{subsection-arbitrary-state}

We now remove the Gaussian assumption on the given invariant state. The invariant-state decomposition provides a faithful Gaussian reference density and shows that it induces the same semigroup on the Hilbert--Schmidt space $ \cH $ in \eqref{eq-definition-mathcal-H}.

\begin{proposition}\label{proposition-state-independence}
Suppose that a finite-mode Gaussian QMS admits a faithful normal invariant state. Then it admits a faithful normal Gaussian invariant state $ \sigma $. For each $ s \in [ 0, 1 ] $, the induced semigroup $ ( T_t^{ ( s ) } ) _{t \geq 0} $ on $ \cH $ is independent of the choice of faithful normal invariant density. In particular, Theorem~\ref{theorem-all-s-representation} gives its second-quantization representation using $ \sigma $.
\end{proposition}
\begin{proof}
By \cite[Theorems~11 and~13]{girotti2025gaussian}, a fixed Gaussian unitary conjugation identifies $ \Gamma ( \C^d ) $ with $ \Gamma ( \C^{d_0} ) \otimes \Gamma ( \C^{d_{-} } ) $, where $ d_0 + d_{-} = d $, and gives
\begin{equation}
    \label{eq-invariant-decomposition}
    \cT_t ( X_0 \otimes X_{-} )
    = \bigl ( \me^{\mi tH_0} X_0 \me^{ - \mi tH_0} \bigr ) \otimes \cT_{-, t} ( X_{-} ).
\end{equation}
Here $ H_0 = \sum_{j = 1}^{d_0} \omega_j N_j $, with $ \omega_j \in \R $, and $ ( \cT_{-, t} ) _{t \geq 0} $ is the Gaussian QMS on $ \cB ( \Gamma ( \C^{d_{-} } ) ) $ whose drift is stable. It has a unique normal invariant state $ \rho_{-} $, which is Gaussian. Every invariant density is $ \rho_0 \otimes \rho_{-} $, where $ \rho_0 $ commutes with $ \me^{\mi tH_0} $ for every $ t \in \R $. Faithfulness of the invariant density implies faithfulness of both factors. A factor with no modes is understood as $ \C $.

This commutation shows that the induced action on the first Hilbert--Schmidt factor is the same unitary conjugation. Hence, for elementary Hilbert--Schmidt tensors,
\begin{equation}
    \label{eq-induced-factorization}
    T_t^{ ( s ) } ( Y_0 \otimes Y_{-} )
    = \bigl ( \me^{\mi tH_0} Y_0 \me^{ - \mi tH_0} \bigr ) \otimes T_{-, t}^{ ( s ) } ( Y_{-} ),
\end{equation}
where $ T_{-, t}^{ ( s ) } $ is induced by $ \cT_{-, t} $ and $ \rho_{-} $. Density extends this identity to the full Hilbert--Schmidt space.

The right-hand side of \eqref{eq-induced-factorization} is independent of $ \rho_0 $. A faithful product thermal density commutes with $ H_0 $, so it can be chosen as $ \rho_0 $. The resulting product $ \rho_0 \otimes \rho_{-} $ is Gaussian and faithful. Undoing the Gaussian unitary conjugation gives the reference density $ \sigma $, to which Theorem~\ref{theorem-all-s-representation} applies.
\end{proof}

\begin{remark}\label{remark-reference-state}
If the given invariant state $ \rho $ is non-Gaussian, Proposition~\ref{proposition-state-independence} provides a faithful Gaussian reference state $ \sigma $ to which Theorem~\ref{theorem-all-s-representation} applies. The matrix $ \mathbf{S}_s $ is computed from $ \sigma $, and $ U_s $ is defined by \eqref{eq-all-s-unitary} with $ \rho $ and $ \chi_\rho $ replaced by $ \sigma $ and $ \chi_\sigma $, respectively. The Gaussian pairing identity is used for this reference state; it is not asserted for the vectors $ \rho^{s/2} W ( z ) \rho^{ ( 1 - s ) /2} $ associated with $ \rho $. By state independence, the resulting representation describes the same induced semigroup. If the drift is stable, uniqueness of the invariant state gives $ \sigma = \rho $.
\end{remark}

\section{Compactness and spectral consequences}\label{section-proof-main}

The second-quantization representation reduces the compactness and spectral questions to properties of the finite-dimensional drift. We first characterize compactness of the induced semigroups and their generator resolvents in terms of drift stability and uniqueness of the normal invariant state, then derive exact trace and norm formulas and determine the generator spectrum with its algebraic multiplicities.

\subsection{Compactness criterion}\label{subsection-compactness-criterion}

By Theorem~\ref{theorem-all-s-representation} and Proposition~\ref{proposition-state-independence}, a faithful normal invariant state ensures that $ \norm{\me^{t \mathbf{Z}_s}} \leq 1 $ for every $ s \in [ 0, 1 ] $ and $ t \geq 0 $. Thus the one-particle evolution does not increase vector norms, but this bound does not imply $ \norm{\me^{t \mathbf{Z}_s}} \to 0 $ as $ t \to \infty $. Since $ \mathbf{Z}_s $ is similar to $ \mathbf{Z} $, this decay is equivalent to stability of the drift. The next theorem characterizes this stronger property through compactness of the induced semigroups and their generator resolvents. It also shows that compactness for a single embedding already determines compactness for the entire family.

\begin{theorem}\label{theorem-main}
Let $ ( \cT_t ) _{t \geq 0} $ be a Gaussian QMS on $ \cB ( \Gamma ( \C^d ) ) $, with $ 1 \leq d < \infty $, admitting a faithful normal invariant state $ \rho $. The following conditions are equivalent:
\begin{enumerate}[label=\textup{(\roman*)},leftmargin=*]
\item\label{condition-stable} The drift matrix $ \mathbf{Z} $ is stable.
\item\label{condition-unique} The semigroup has a unique normal invariant state.
\item\label{condition-kernel} For some $ s \in [ 0, 1 ] $, $ \ker L^{ ( s ) } $ is finite dimensional.
\item\label{condition-resolvent} For some $ s \in [ 0, 1 ] $, $ L^{ ( s ) } $ has compact resolvent.
\item\label{condition-one-time} For some $ s \in [ 0, 1 ] $ and some $ t > 0 $, $ T_t^{ ( s ) } $ is compact.
\item\label{condition-immediate} For every $ s \in [ 0, 1 ] $ and every $ t > 0 $, $ T_t^{ ( s ) } $ is compact.
\item\label{condition-trace-class} For every $ s \in [ 0, 1 ] $ and every $ t > 0 $, $ T_t^{ ( s ) } \in \cS_1 ( \cH ) $.
\end{enumerate}
If these conditions hold, $ \rho $ is Gaussian and $ \ker L^{ ( s ) } = \C \rho^{1/2} $ for every $ s $. If they fail, every $ \ker L^{ ( s ) } $ is infinite dimensional.
\end{theorem}

The equivalence of \ref{condition-stable} and \ref{condition-unique} was established by Girotti and Poletti \cite[Theorems~11 and~13]{girotti2025gaussian}; Theorem~\ref{theorem-main} connects these conditions with the induced operators. In particular, ``for some $ s $'' in \ref{condition-kernel} and \ref{condition-resolvent} can be replaced by ``for every $ s $''. If the drift is not stable, no induced generator has compact resolvent.

The second-quantization representation reduces compactness to strict contractivity of the one-particle semigroup. The following lemma isolates the obstruction caused by imaginary drift eigenvalues.

\begin{lemma}\label{lemma-infinite-fixed-space}
Under the assumptions of Theorem~\ref{theorem-main}, if $ \mathbf{Z} $ is not stable, every $ \ker L^{ ( s ) } $ is infinite dimensional.
\end{lemma}
\begin{proof}
By Proposition~\ref{proposition-state-independence}, we may use a faithful Gaussian invariant reference state and apply Theorem~\ref{theorem-all-s-representation}. Fix $ s \in [ 0, 1 ] $. Applying the contraction $ \me^{t \mathbf{Z}_s} $ to an eigenvector shows that every eigenvalue $ \lambda $ of $ \mathbf{Z}_s $ satisfies $ \Re \lambda \leq 0 $. Since $ \mathbf{Z}_s $ is similar to $ \mathbf{Z} $, failure of stability gives an eigenvalue $ \mi \omega $ with $ \omega \in \R $.

By Corollary~\ref{corollary-generator-domain}, every finite-particle vector $ f $ belongs to the domain of $ \lambda ( \mathbf{Z}_s ) $, and
\begin{equation*}
    L^{ ( s ) } U_s f = U_s \lambda ( \mathbf{Z}_s ) f.
\end{equation*}
Thus it suffices to construct infinitely many orthogonal finite-particle vectors annihilated by $ \lambda ( \mathbf{Z}_s ) $.

If $ \omega = 0 $, choose a nonzero vector $ u $ with $ \mathbf{Z}_s u = 0 $. For $ n \geq 1 $, put $ v_n = u^{ \otimes n} \in ( \C^{2d} ) ^{ \odot n} $. The tensor-factor formula in \eqref{eq-second-quantized-generator} gives
\begin{equation*}
    \lambda ( \mathbf{Z}_s ) v_n
    = \sum_{j = 1}^n u^{ \otimes ( j - 1 ) } \otimes ( \mathbf{Z}_s u ) \otimes u^{ \otimes ( n - j ) }
    = 0.
\end{equation*}
If $ \omega \neq 0 $, reality of $ \mathbf{Z} $ and similarity of $ \mathbf{Z}_s $ to $ \mathbf{Z} $ imply that $ - \mi \omega $ is also an eigenvalue. Choose nonzero eigenvectors $ u_{+} $ and $ u_{-} $ such that $ \mathbf{Z}_s u_{+} = \mi \omega u_{+} $ and $ \mathbf{Z}_s u_{-} = - \mi \omega u_{-} $. Set
\begin{equation*}
    v_n = u_{+} ^{ \odot n} \odot u_{-} ^{ \odot n} \in ( \C^{2d} ) ^{ \odot 2n}, \qquad n \geq 1.
\end{equation*}
Differential second quantization applies $ \mathbf{Z}_s $ to each of the $ 2n $ factors and sums the results. The $ n $ factors $ u_{+} $ contribute $ n \mi \omega v_n $, while the $ n $ factors $ u_{-} $ contribute $ - n \mi \omega v_n $. Hence
\begin{equation*}
    \lambda ( \mathbf{Z}_s ) v_n = n \mi \omega v_n - n \mi \omega v_n = 0.
\end{equation*}
Since $ u_{+} $ and $ u_{-} $ are nonzero, their symmetric tensor product $ v_n $ is nonzero.
The vacuum is also annihilated by $ \lambda ( \mathbf{Z}_s ) $. In both cases, the vectors $ v_n $ lie in distinct particle subspaces, so normalization gives infinitely many orthonormal vectors. Their Hilbert--Schmidt images satisfy
\begin{equation*}
    U_s v_n \in \dom L^{ ( s ) }, \qquad
    L^{ ( s ) } U_s v_n = U_s \lambda ( \mathbf{Z}_s ) v_n = 0.
\end{equation*}
After normalization, their images under $ U_s $ form an orthonormal sequence in $ \ker L^{ ( s ) } $.
\end{proof}

\begin{proof}[Proof of Theorem~\ref{theorem-main}]
By \cite[Theorems~11 and~13]{girotti2025gaussian}, conditions \ref{condition-stable} and \ref{condition-unique} are equivalent, and under \ref{condition-stable} the invariant state $ \rho $ is Gaussian. We first prove that \ref{condition-stable} implies \ref{condition-immediate} and \ref{condition-trace-class}. Assume that the drift is stable and fix $ s \in [ 0, 1 ] $. By Theorem~\ref{theorem-all-s-representation}, the induced semigroup is unitarily equivalent to $ \Gamma ( \me^{t \mathbf{Z}_s} ) $. The matrices $ \me^{t \mathbf{Z}_s} $ are contractions and converge to zero because $ \mathbf{Z}_s $ is similar to the stable matrix $ \mathbf{Z} $.

In fact, $ \norm{\me^{t \mathbf{Z}_s}} < 1 $ for every $ t > 0 $. Otherwise, finite dimensionality gives a unit vector $ v $ and $ t_0 > 0 $ such that $ \abs{\me^{t_0 \mathbf{Z}_s} v} = 1 $. Contractivity and the semigroup property imply
\begin{equation*}
    1 = \abs{\me^{t_0 \mathbf{Z}_s} v} \leq \abs{\me^{t \mathbf{Z}_s} v} \leq 1,
    \qquad 0 \leq t \leq t_0.
\end{equation*}
The real-analytic function $ t \mapsto \abs{\me^{t \mathbf{Z}_s} v}^2 $ is therefore constant on an interval, and hence on $ \R $, contradicting its convergence to zero.

Fix $ t > 0 $. For $ r \in \N $, let $ Q_{\leq r} $ be the orthogonal projection onto $ \bigoplus_{n = 0}^r ( \C^{2d} ) ^{ \odot n} $. Since $ \Gamma ( \me^{t \mathbf{Z}_s} ) $ acts as $ ( \me^{t \mathbf{Z}_s} ) ^{ \odot n} $ on the $ n $-particle subspace, subtracting its truncation gives
\begin{equation*}
    \Gamma ( \me^{t \mathbf{Z}_s} ) - \Gamma ( \me^{t \mathbf{Z}_s} ) Q_{\leq r}
    = \underbrace{0 \oplus \cdots \oplus 0}_{n = 0, \ldots, r}
    \oplus \bigoplus_{n = r + 1}^\infty ( \me^{t \mathbf{Z}_s} ) ^{ \odot n}.
\end{equation*}
The norm of this orthogonal direct sum is the supremum of its block norms. The symmetric tensor norm identity in \cite[Lemma~3.1]{ArakiSinhaSunder1995} and strict contractivity therefore give
\begin{equation}
    \label{eq-chaos-tail}
        \norm{\Gamma ( \me^{t \mathbf{Z}_s} ) - \Gamma ( \me^{t \mathbf{Z}_s} ) Q_{\leq r}} = \sup_{n \geq r + 1} \norm{( \me^{t \mathbf{Z}_s} ) ^{ \odot n}} = \norm{\me^{t \mathbf{Z}_s}}^{r + 1} \longrightarrow 0, \qquad r \to \infty.
\end{equation}
Since $ Q_{\leq r} $ has finite rank and $ s $ and $ t > 0 $ were arbitrary, this proves \ref{condition-immediate}. Moreover, $ \me^{t \mathbf{Z}_s} $ is trace class on $ \C^{2d} $, so \cite[Theorem~3.10(a)]{ArakiSinhaSunder1995} proves \ref{condition-trace-class}.

We next establish the remaining implications. Trace-class membership implies compactness, so \ref{condition-trace-class} implies \ref{condition-immediate}, which implies \ref{condition-one-time}. By \cite[Theorem~II.4.29]{EngelNagel2000}, immediate compactness gives compact resolvent, so \ref{condition-immediate} also implies \ref{condition-resolvent}.

Suppose that \ref{condition-one-time} holds, and choose the corresponding $ s $ and $ t > 0 $. The restriction of $ T_t^{ ( s ) } $ to $ \ker L^{ ( s ) } $ is the identity, so compactness forces this kernel to be finite dimensional. Thus \ref{condition-one-time} implies \ref{condition-kernel}. Similarly, if \ref{condition-resolvent} holds for some $ s $, then, for $ \lambda > 0 $, the restriction of $ ( \lambda - L^{ ( s ) } ) ^{ - 1} $ to the kernel is $ \lambda^{ - 1}\1 $. A nonzero scalar multiple of the identity is compact only on a finite-dimensional space, so \ref{condition-resolvent} implies \ref{condition-kernel}. Finally, Lemma~\ref{lemma-infinite-fixed-space} gives \ref{condition-kernel} $ \Rightarrow $ \ref{condition-stable}. Together with the cited equivalence of \ref{condition-stable} and \ref{condition-unique}, this proves all seven conditions equivalent.

It remains to identify the kernels. Assume that the drift is stable and fix any $ s \in [ 0, 1 ] $. The vacuum vector $ e ( 0 ) \in \Gamma ( \C^{2d} ) $ is fixed by second quantization. Since $ \norm{\me^{t \mathbf{Z}_s}} \leq 1 $, we have
\begin{equation*}
    \norm{\Gamma ( \me^{t \mathbf{Z}_s} ) - \ketbra{e ( 0 )}{e ( 0 )}}
    = \sup_{n \geq 1} \norm{\me^{t \mathbf{Z}_s} }^n = \norm{\me^{t \mathbf{Z}_s} } \longrightarrow 0.
\end{equation*}
Since $ U_s e ( 0 ) = \rho^{1/2} $, the limiting projection on $ \cH $ is $ \ketbra{\rho^{1/2}}{\rho^{1/2}} $. Every vector in $ \ker L^{ ( s ) } $ is fixed by $ T_t^{ ( s ) } $ for all $ t \geq 0 $; passing to the limit as $ t \to \infty $ shows that it belongs to $ \C \rho^{1/2} $. Conversely, $ \rho^{1/2} $ is fixed by the semigroup and hence belongs to the generator kernel. Thus $ \ker L^{ ( s ) } = \C \rho^{1/2} $ for every $ s $. In the nonstable case, every kernel is infinite dimensional by Lemma~\ref{lemma-infinite-fixed-space}.
\end{proof}

\begin{remark}\label{remark-no-general-equivalence}
Reality of the Gaussian drift is essential in the proof of necessity. For a general complex one-particle generator, compact resolvent need not imply compactness of the semigroup. For example, on $ \Gamma ( \C ) $, the generator $ \mi N = \lambda ( \mi ) $ has compact resolvent because $ ( \eta - \mi n ) ^{ - 1} \to 0 $ for every $ \eta > 0 $, whereas $ \Gamma ( \me^{\mi t} ) $ is unitary and noncompact. For a real Gaussian drift, each imaginary eigenvalue has its conjugate, which produces the infinitely many fixed tensors used in Lemma~\ref{lemma-infinite-fixed-space}.
\end{remark}

\begin{example}[A faithful invariant state without uniqueness]\label{example-nonunique}
On one mode, let $ H = \omega N $, where $ \omega \in \R $, and set
$ \cT_t ( X ) = \me^{\mi tH} X \me^{ - \mi tH} $. This Gaussian QMS has drift $ Zz = \mi \omega z $ and zero diffusion. Every thermal density $ \rho_{\beta} $ in \eqref{eq-thermal-density} is a faithful normal Gaussian invariant state. More generally, every normal state diagonal in the number basis is invariant.

For the one-mode number vectors $ \ket{n} $, $ n \in \N $, we have $ H \ket{n} = \omega n \ket{n} $, hence $ \cT_t ( \ketbra{n}{n} ) = \ketbra{n}{n} $. For the embedding associated with $ \rho_{\beta} $,
\begin{equation*}
    \rho_{\beta}^{s/2} \ketbra{n}{n} \rho_{\beta}^{ ( 1 - s ) /2} = \sqrt{1 - \me^{ - \beta}} \me^{ - \beta n/2} \ketbra{n}{n}.
\end{equation*}
The coefficient is positive, so \eqref{eq-construction-induced-semigroup} gives $ T_t^{ ( s ) } \ketbra{n}{n} = \ketbra{n}{n} $ for every $ t \geq 0 $ and $ s \in [ 0, 1 ] $. Moreover, $ \langle \ketbra{m}{m}, \ketbra{n}{n} \rangle_2 = \delta_{mn} $, so $ ( \ketbra{n}{n} ) _{n \in \N} $ is an orthonormal sequence in $ \cH $. For $ m \neq n $,
\begin{equation*}
    \norm{\ketbra{n}{n} - \ketbra{m}{m}}_2 = \sqrt{2}.
\end{equation*}
Thus the image of this bounded sequence has no norm-convergent subsequence. A compact operator maps every bounded sequence to one with a convergent subsequence, so $ T_t^{ ( s ) } $ is noncompact for every $ t > 0 $.

These fixed vectors belong to the generator domain and satisfy $ L^{ ( s ) } \ketbra{n}{n} = 0 $. Hence, for $ \eta > 0 $,
\begin{equation*}
    ( \eta \1 - L^{ ( s ) } ) ^{ - 1} \ketbra{n}{n} = \eta^{ - 1} \ketbra{n}{n}.
\end{equation*}
The resolvent is noncompact for the same reason. Tensoring this oscillator with a stable dissipative Gaussian QMS with faithful invariant density $ \rho_{-} $ preserves the obstruction: the vectors $ \ketbra{n}{n} \otimes \rho_{-}^{1/2} $ form an orthonormal sequence fixed by the induced semigroup in \eqref{eq-induced-factorization}. Both constructions are elementary instances of Proposition~\ref{proposition-state-independence}.
\end{example}

\subsection{Trace and norm formulas}\label{section-traces}

Under the equivalent conditions of Theorem~\ref{theorem-main}, the strict one-particle contractions give exact trace-norm and Hilbert--Schmidt norm formulas. They also determine the trace, independently of the embedding parameter.

\begin{corollary}\label{corollary-traces}
Assume the equivalent conditions of Theorem~\ref{theorem-main}. For $ t > 0 $ and $ s \in [ 0, 1 ] $, we have
\begin{align}
    \norm{T_t^{ ( s ) }}_{\cS_1}
    & = \det ( \1_{2d} - \abs{\me^{t \mathbf{Z}_s}} ) ^{ - 1}, \label{eq-all-s-schatten}\\*
    \norm{T_t^{ ( s ) }}_{\cS_2}^{2}
    & = \det \bigl ( \1_{2d} - ( \me^{t \mathbf{Z}_s} ) ^* \me^{t \mathbf{Z}_s} \bigr ) ^{ - 1}, \label{eq-all-s-hilbert-schmidt}\\*
    \tr_{\cH} ( T_t^{ ( s ) } )
    & = \det ( \1_{2d} - \me^{t \mathbf{Z}} ) ^{ - 1}.\label{eq-all-s-trace}
\end{align}
Moreover, the operator norm on $ \cH $ satisfies
\begin{equation}
    \label{eq-all-s-centered-bound}
    \norm{T_t^{ ( s ) } - \ketbra{\rho^{1/2}}{\rho^{1/2}}} = \norm{\me^{t \mathbf{Z}_s}} < 1, \qquad t > 0.
\end{equation}
\end{corollary}
\begin{proof}
By Theorem~\ref{theorem-all-s-representation}, the two norm formulas follow from \cite[Corollary~3.8]{ArakiSinhaSunder1995}. The trace formula in \cite[Remark~3.7(1)]{ArakiSinhaSunder1995} initially gives $ \det ( \1_{2d} - \me^{t \mathbf{Z}_s} ) ^{ - 1} $. Equation~\eqref{eq-definition-A-s} shows that $ \me^{t \mathbf{Z}_s} $ and $ \me^{t \mathbf{Z}} $ are similar, so
\begin{equation*}
    \det ( \1_{2d} - \me^{t \mathbf{Z}_s} ) = \det ( \1_{2d} - \me^{t \mathbf{Z}} ).
\end{equation*}
This gives the trace identity in terms of the original drift $ \mathbf{Z} $. Equation~\eqref{eq-all-s-centered-bound} is the tensor-norm calculation in the proof of Theorem~\ref{theorem-main}.
\end{proof}

The trace norm and Hilbert--Schmidt norm involve the singular values of $ \me^{t \mathbf{Z}_s} $ and generally depend on $ s $. The trace depends only on the eigenvalues of the drift. The same one-particle matrix also determines the generator spectrum.

\subsection{Spectrum and algebraic multiplicities}\label{subsection-generator-spectrum}

For one mode, the following spectral formula recovers the full spectrum established in \cite[Sections~3 and~5]{fagnola_li_2025spectral}. The spectral formula for symmetric second quantization is already known from van Neerven's work \cite[Theorem~3.3 and Corollary~4.4, with $ p = 2 $]{vanNeerven2005}. We reproduce the finite-dimensional calculation in our notation for completeness and readability, and spell out the elementary count of algebraic multiplicities.

\begin{corollary}[Spectrum and algebraic multiplicities]\label{corollary-spectrum}
Assume the equivalent conditions of Theorem~\ref{theorem-main}, and let $ \lambda_1, \ldots, \lambda_{2d} $ be the eigenvalues of $ \mathbf{Z} $, repeated according to algebraic multiplicity. Then, for every $ s \in [ 0, 1 ] $,
\begin{equation}
    \label{eq-all-s-generator-spectrum}
    \sigma ( L^{ ( s ) } ) =
    \left\{\sum_{j = 1}^{2d} \alpha_j \lambda_j : \alpha \in \N^{2d} \right\}.
\end{equation}
Every spectral value is an isolated eigenvalue of finite algebraic multiplicity. The algebraic multiplicity of $ \mu $ is
\begin{equation}
    \label{eq-algebraic-multiplicity}
    \#\left\{\alpha \in \N^{2d} : \sum_{j = 1}^{2d} \alpha_j \lambda_j = \mu \right\}.
\end{equation}
In particular, the spectrum and its algebraic multiplicities are independent of $ s $. The matrix $ \mathbf{Z} $ need not be diagonalizable.
\end{corollary}
\begin{proof}
Fix $ s \in [ 0, 1 ] $. Corollary~\ref{corollary-generator-domain} identifies the generator on Fock space:
\begin{equation*}
    U_s^* L^{ ( s ) } U_s = \lambda ( \mathbf{Z}_s ) = \bigoplus_{n = 0}^\infty \mathbf{Z}_{s, n}.
\end{equation*}
We first explain why these finite-dimensional blocks determine the full spectrum. Compact resolvent, established in Theorem~\ref{theorem-main}, implies that every spectral value is an isolated eigenvalue of finite algebraic multiplicity. For a vector $ f = \bigoplus_{n = 0}^\infty f_n $ in the generator domain, the eigenvalue equation is componentwise:
\begin{equation*}
    \lambda ( \mathbf{Z}_s ) f = \mu f
    \quad \Longleftrightarrow \quad
    \mathbf{Z}_{s, n} f_n = \mu f_n \quad \text{for every } n \geq 0.
\end{equation*}
If $ f \neq 0 $, at least one $ f_n $ is nonzero, so $ \mu $ is an eigenvalue of that block. Conversely, an eigenvector of a block, placed in its particle component with all other components zero, is a finite-particle vector. It belongs to the full generator domain and satisfies the same eigenvalue equation. Unitary equivalence therefore gives
\begin{equation*}
    \sigma ( L^{ ( s ) } ) = \bigcup_{n = 0}^\infty \sigma ( \mathbf{Z}_{s, n} ).
\end{equation*}

We next compute the eigenvalues and algebraic multiplicities of each block. Since $ \mathbf{Z}_s $ is similar to $ \mathbf{Z} $, its eigenvalues are $ \lambda_1, \ldots, \lambda_{2d} $. Schur's triangularization theorem gives an orthonormal basis $ u_1, \ldots, u_{2d} $ in which
\begin{equation*}
    \mathbf{Z}_s u_j = \lambda_j u_j + \sum_{i = 1}^{j - 1} \langle u_i, \mathbf{Z}_s u_j \rangle u_i, \qquad j = 1, \ldots, 2d.
\end{equation*}
The coefficients $ \langle u_i, \mathbf{Z}_s u_j \rangle $, with $ i < j $, are the entries above the diagonal of the matrix of $ \mathbf{Z}_s $ in this basis.
For $ n \geq 1 $, a basis of $ ( \C^{2d} ) ^{ \odot n} $ consists of the symmetric products
\begin{equation*}
    u_1^{ \odot \alpha_1} \odot \cdots \odot u_{2d}^{ \odot \alpha_{2d}},
    \qquad \alpha \in \N^{2d}, \quad \abs{\alpha} := \sum_{j = 1}^{2d} \alpha_j = n.
\end{equation*}
The tensor-factor sum defining $ \mathbf{Z}_{s, n} $ applies $ \mathbf{Z}_s $ to one factor at a time. The $ \alpha_j $ factors equal to $ u_j $ contribute $ \alpha_j \lambda_j $ to the coefficient of the original basis vector. Every other term replaces one $ u_j $ by some $ u_i $ with $ i < j $. Order the basis by increasing $ \sum_{j = 1}^{2d} j \alpha_j $, breaking ties arbitrarily. Each such replacement strictly decreases this quantity, so all the other terms involve earlier basis vectors. Thus the block matrix is upper triangular, with diagonal entry $ \sum_{j = 1}^{2d} \alpha_j \lambda_j $ at the vector indexed by $ \alpha $. Its characteristic polynomial is consequently
\begin{equation*}
    \det ( z \1 - \mathbf{Z}_{s, n} )
    = \prod_{\substack{\alpha \in \N^{2d} \\ \abs{\alpha} = n}}
    \left( z - \sum_{j = 1}^{2d} \alpha_j \lambda_j \right).
\end{equation*}
The same formula holds for $ n = 0 $: the vacuum subspace is one dimensional and its generator is zero. Taking the union of the block spectra proves \eqref{eq-all-s-generator-spectrum}. The characteristic polynomial also shows that the algebraic multiplicity of $ \mu $ in the $ n $-particle block equals the number of multi-indices $ \alpha $ with $ \abs{\alpha} = n $ and $ \sum_j \alpha_j \lambda_j = \mu $.

It remains to show that these block multiplicities add up to the algebraic multiplicity for the full generator. Fix a spectral value $ \mu $ and put $ \delta := \min_j ( - \Re \lambda_j ) > 0 $. If a multi-index $ \alpha $ contributes to $ \mu $ in the $ n $-particle block, then
\begin{equation*}
    - \Re \mu = \sum_{j = 1}^{2d} \alpha_j ( - \Re \lambda_j ) \geq \delta \abs{\alpha} = \delta n.
\end{equation*}
Hence only finitely many particle numbers can contribute to $ \mu $.

Applying the direct-sum generator formula repeatedly, any generalized eigenvector $ f = \bigoplus_{n = 0}^\infty f_n $ satisfying $ ( \lambda ( \mathbf{Z}_s ) - \mu \1 ) ^k f = 0 $ for some $ k \geq 1 $ has
\begin{equation*}
    ( \mathbf{Z}_{s, n} - \mu \1 ) ^k f_n = 0, \qquad n \geq 0.
\end{equation*}
If $ \mu $ is not an eigenvalue of a block, the corresponding matrix $ \mathbf{Z}_{s, n} - \mu \1 $ is invertible, and therefore $ f_n = 0 $. Thus every generalized eigenvector has components only in the finitely many contributing blocks. Conversely, a finite sum of block generalized eigenvectors is a finite-particle vector, belongs to every generator-power domain, and is annihilated by a sufficiently large common power of $ \lambda ( \mathbf{Z}_s ) - \mu \1 $. The full generalized eigenspace is therefore the direct sum of the generalized eigenspaces of these blocks. Their dimensions are the block algebraic multiplicities computed above, so adding them proves \eqref{eq-algebraic-multiplicity}. Finally, both formulas depend only on the eigenvalues of $ \mathbf{Z} $, and hence are independent of $ s $. Schur's theorem applies to every complex matrix, so no diagonalizability assumption is needed.
\end{proof}

The one-mode analysis in \cite{fagnola_li_2025spectral} gives more detailed information in the original operator and phase-space variables. Its polynomial filtration constructs eigenvectors and Jordan chains from polynomials in the creation and annihilation operators, while its characteristic-function representation proves regularization and immediate compactness. It also determines the spectrum of the self-adjoint closure of the symmetric part and the exact coercive spectral gap. These results explain the eigenvectors, regularity, and dissipation that are not specified by the eigenvalue list alone. The present Fock-space argument complements that analysis by giving a common structural proof of the spectral formula for every finite number of modes.

\section{Coercive spectral gaps and compactness} \label{section-gaps-compactness}

This section is motivated by the one-mode $ L^2 $ analysis in \cite[Proposition~4.8]{fagnola_li_2025spectral}. Under its standing assumptions of drift stability and a faithful normal invariant state, a positive spectral gap implies compact resolvent of the induced generator and its adjoint. We revisit this implication for finitely many modes under uniqueness of the normal invariant state. Using the second-quantization representation, Proposition~\ref{proposition-gap-resolvent} gives a direct proof through estimates on the particle subspaces and yields an explicit bound for finite-rank resolvent approximation.

We also show that uniqueness cannot be dropped when the gap is measured on the orthogonal complement of the full generator kernel. Example~\ref{example-gap-infinite-kernel} provides a Gaussian QMS with nonunique faithful normal invariant states and a positive coercive spectral gap for every embedding, while its induced generators have noncompact resolvents.

Throughout this section the Gaussian QMS has a faithful normal invariant state. If this state is not Gaussian, $ U_s $ and $ \mathbf{Z}_s $ are constructed from the Gaussian reference state in Proposition~\ref{proposition-state-independence}.

\subsection{Spectral gaps and decay estimates}\label{subsection-gaps-compactness}

We measure the gap on the orthogonal complement of the full generator kernel so that the definition also applies when the invariant state is not unique. For the generator $ L $ of a contraction semigroup on a Hilbert space, define
\begin{equation}
    \label{eq-definition-coercive-gap}
    \gap L := \inf \left\{ - \Re \langle Y, LY \rangle :
    Y \in \dom L \cap ( \ker L ) ^ \perp, \ \norm{Y} = 1 \right\}.
\end{equation}
We say that $ L $ has a \emph{spectral gap} when $ \gap L > 0 $. In general, this coercive spectral gap differs from the distance of the nonzero spectrum to the imaginary axis. For a self-adjoint generator, the two quantities agree. Indeed, $ L \leq 0 $, and $ ( \ker L ) ^\perp $ is a reducing subspace. The spectral theorem identifies the infimum of the Rayleigh quotients of $ - L $ on this subspace with the bottom of the spectrum of its restriction. Consequently,
\begin{equation*}
    \gap L = \inf \{ - \lambda : \lambda \in \sigma ( L ) \setminus \{ 0 \} \}.
\end{equation*}
Since $ \sigma ( L ) \subset ( - \infty, 0 ] $, this is the distance of the nonzero spectrum from the imaginary axis. If nonzero spectral values accumulate at zero, both quantities vanish.

Suppose first that the Gaussian QMS has a unique faithful normal invariant state. By Theorem~\ref{theorem-main}, $ \ker L^{ ( s ) } = \C \rho^{1/2} $ for every $ s $. In this case, $ \gap L^{ ( s ) } \geq c $ is equivalent to
\begin{equation}
    \label{eq-gap-prefactor-one}
    \norm{T_t^{ ( s ) } Y}_2 \leq \me^{ - ct} \norm{Y}_2,
    \qquad t \geq 0, \quad Y \perp \rho^{1/2}.
\end{equation}
Indeed, the unit vector $ \rho^{1/2} $ is fixed by each contraction $ T_t^{ ( s ) } $ and therefore also by its adjoint. Its orthogonal complement is consequently invariant. On the generator domain, differentiation of the squared norm and integration of the resulting differential inequality give \eqref{eq-gap-prefactor-one}; density then extends the estimate to the whole orthogonal complement. Conversely, differentiation of \eqref{eq-gap-prefactor-one} at zero gives the coercive bound.

Define the Hermitian matrix
\begin{equation}
    \label{eq-drift-dissipative-part}
    \mathbf{K}_s := - \frac{\mathbf{Z}_s + \mathbf{Z}_s^*}{2} \geq 0.
\end{equation}
Positivity follows from contractivity of the one-particle semigroup in Theorem~\ref{theorem-all-s-representation}.

\begin{proposition}\label{proposition-gap-formula}
If $ \mathbf{Z} $ is stable, then for every $ s \in [ 0, 1 ] $,
\begin{equation}
    \label{eq-stable-gap-value}
    \gap L^{ ( s ) } = \min \sigma ( \mathbf{K}_s ).
\end{equation}
\end{proposition}
\begin{proof}
Fix $ s \in [ 0, 1 ] $ and put $ g_s := \min \sigma ( \mathbf{K}_s ) \geq 0 $. Since $ \mathbf{K}_s $ is Hermitian, $ \mathbf{K}_s \geq g_s \1_{2d} $. For $ n \geq 1 $, both the full tensor-factor sum defining $ \mathbf{Z}_{s, n} $ in Corollary~\ref{corollary-generator-domain} and its adjoint commute with permutations of the factors. The symmetric subspace is therefore invariant under both operators, so taking the adjoint commutes with restriction to that subspace. It follows that
\begin{equation*}
    - \frac{\mathbf{Z}_{s, n} + \mathbf{Z}_{s, n}^*}{2}
    = \left.\sum_{j = 1}^n \1^{ \otimes ( j - 1 ) } \otimes \mathbf{K}_s \otimes \1^{ \otimes ( n - j ) }\right|_{( \C^{2d} ) ^{ \odot n}}
    \geq n g_s \1.
\end{equation*}
In the tensor products, each $ \1 $ is the identity on $ \C^{2d} $; the $ \1 $ on the right-hand side of the displayed inequality is the identity on $ ( \C^{2d} ) ^{ \odot n} $. Before restriction, each summand is bounded below by $ g_s $ times the identity on the full tensor product $ ( \C^{2d} ) ^{ \otimes n} $. Summing these $ n $ inequalities and restricting gives the displayed bound. Thus, for every $ f_n \in ( \C^{2d} ) ^{ \odot n} $,
\begin{equation*}
    - \Re \langle f_n, \mathbf{Z}_{s, n} f_n \rangle
    = \left\langle f_n, - \frac{\mathbf{Z}_{s, n} + \mathbf{Z}_{s, n}^*}{2} f_n \right\rangle
    \geq n g_s \norm{f_n}^2.
\end{equation*}

We next pass from the particle subspaces to the full generator domain. By Theorem~\ref{theorem-main}, $ \ker L^{ ( s ) } = \C \rho^{1/2} $, and $ U_s e ( 0 ) = \rho^{1/2} $. Take $ Y \in \dom L^{ ( s ) } \cap ( \ker L^{ ( s ) } ) ^\perp $ and write $ f = U_s^* Y = \bigoplus_{n = 0}^\infty f_n $. Orthogonality to $ \rho^{1/2} $ means that $ f_0 = 0 $. For $ r \geq 1 $, let $ f^{ ( r ) } = \bigoplus_{n = 0}^r f_n $, with all higher particle components zero. These finite-particle vectors remain orthogonal to the vacuum. Orthogonality of the particle subspaces and the generator formula in Corollary~\ref{corollary-generator-domain} give
\begin{equation*}
    - \Re \langle U_s f^{ ( r ) }, L^{ ( s ) } U_s f^{ ( r ) } \rangle_2
    = \sum_{n = 1}^r - \Re \langle f_n, \mathbf{Z}_{s, n} f_n \rangle
    \geq g_s \sum_{n = 1}^r n \norm{f_n}^2
    \geq g_s \sum_{n = 1}^r \norm{f_n}^2 = g_s \norm{f^{ ( r ) }}^2.
\end{equation*}
Here the second inequality uses $ n \geq 1 $ and $ g_s \geq 0 $. The domain formula in Corollary~\ref{corollary-generator-domain} also gives
\begin{equation*}
    \norm{U_s f^{ ( r ) } - Y}_2^2
    + \norm{L^{ ( s ) } U_s f^{ ( r ) } - L^{ ( s ) } Y}_2^2
    = \sum_{n > r} \left( \norm{f_n}^2 + \norm{\mathbf{Z}_{s, n} f_n}^2 \right) \longrightarrow 0.
\end{equation*}
Thus both the vectors and their generator images converge in norm. Passing to the limit in the preceding scalar-product inequality yields
\begin{equation*}
    - \Re \langle Y, L^{ ( s ) } Y \rangle_2 \geq g_s \norm{Y}_2^2.
\end{equation*}
Taking the infimum over unit vectors $ Y $ in the definition \eqref{eq-definition-coercive-gap} proves $ \gap L^{ ( s ) } \geq g_s $.

To prove the reverse inequality, choose a unit eigenvector $ u \in \C^{2d} $ of the Hermitian matrix $ \mathbf{K}_s $ with $ \mathbf{K}_s u = g_s u $. Regard $ u $ as a Fock vector whose only nonzero component is its one-particle component. It is a finite-particle vector and is orthogonal to the vacuum. Hence $ U_s u $ belongs to $ \dom L^{ ( s ) } \cap ( \ker L^{ ( s ) } ) ^\perp $ and has norm one. Since $ \mathbf{Z}_{s, 1} = \mathbf{Z}_s $, we obtain
\begin{equation*}
    - \Re \langle U_s u, L^{ ( s ) } U_s u \rangle_2
    = - \Re \langle u, \mathbf{Z}_s u \rangle
    = \langle u, \mathbf{K}_s u \rangle = g_s.
\end{equation*}
This admissible vector gives $ \gap L^{ ( s ) } \leq g_s $, proving the asserted equality. The argument also applies when $ g_s = 0 $.
\end{proof}

\begin{proposition}[A quantitative resolvent estimate]\label{proposition-gap-resolvent}
Suppose that the normal invariant state is unique, and fix $ s \in [ 0, 1 ] $ such that $ g_s := \gap L^{ ( s ) } > 0 $. For every $ \eta > 0 $ and $ r \in \N $, the resolvent $ ( \eta \1 - L^{ ( s ) } ) ^{ - 1} $ admits a finite-rank approximation obtained by keeping particle numbers $ 0, \ldots, r $, whose difference from the resolvent has operator norm at most $ ( \eta + ( r + 1 ) g_s ) ^{ - 1} $. In particular, $ L^{ ( s ) } $ and its adjoint have compact resolvent.
\end{proposition}

\begin{proof}
By Theorem~\ref{theorem-main}, $ \mathbf{Z} $ is stable. Proposition~\ref{proposition-gap-formula} therefore gives $ \mathbf{K}_s \geq g_s \1_{2d} $, or equivalently
\begin{equation}
    \label{eq-gap-one-particle-coercivity}
    - \Re \langle \xi, \mathbf{Z}_s \xi \rangle \geq g_s \abs{\xi}^2, \qquad \xi \in \C^{2d}.
\end{equation}
The tensor-factor formula \eqref{eq-second-quantized-generator} then yields
\begin{equation}
    \label{eq-gap-particle-coercivity}
    - \Re \langle f_n, \mathbf{Z}_{s, n} f_n \rangle \geq n g_s \norm{f_n}^2, \qquad f_n \in ( \C^{2d} ) ^{ \odot n}.
\end{equation}
Fix $ \eta > 0 $. For $ n \geq 1 $, the preceding inequality and Cauchy--Schwarz give
\begin{equation*}
    ( \eta + n g_s ) \norm{f_n}^2
    \leq \Re \langle f_n, ( \eta \1 - \mathbf{Z}_{s, n} ) f_n \rangle
    \leq \abs{\langle f_n, ( \eta \1 - \mathbf{Z}_{s, n} ) f_n \rangle}
    \leq \norm{f_n} \, \norm{( \eta \1 - \mathbf{Z}_{s, n} ) f_n}.
\end{equation*}
For $ f_n \neq 0 $, division by $ \norm{f_n} $ yields $ \norm{( \eta \1 - \mathbf{Z}_{s, n} ) f_n} \geq ( \eta + n g_s ) \norm{f_n} $. Hence $ \eta \1 - \mathbf{Z}_{s, n} $ is injective and, since the particle subspace is finite dimensional, invertible. Substituting $ f_n = ( \eta \1 - \mathbf{Z}_{s, n} ) ^{ - 1} h_n $ gives
\begin{equation*}
    \norm{( \eta \1 - \mathbf{Z}_{s, n} ) ^{ - 1} h_n} \leq \frac{\norm{h_n}}{\eta + n g_s}, \qquad
    \norm{( \eta \1 - \mathbf{Z}_{s, n} ) ^{ - 1}} \leq \frac{1}{\eta + n g_s}.
\end{equation*}
These estimates also hold for $ n = 0 $, since $ \mathbf{Z}_{s, 0} = 0 $ and its resolvent is $ \eta^{ - 1} \1 $.

Let $ R_{\eta, s} := ( \eta \1 - L^{ ( s ) } ) ^{ - 1} $. We next justify its componentwise representation using the domain in Corollary~\ref{corollary-generator-domain}. For $ h = \bigoplus_{n = 0}^{\infty} h_n \in \Gamma ( \C^{2d} ) $, put $ f_n = ( \eta \1 - \mathbf{Z}_{s, n} ) ^{ - 1} h_n $. The uniform bound $ \norm{f_n} \leq \eta^{ - 1} \norm{h_n} $ gives $ f = \bigoplus_{n = 0}^{\infty} f_n \in \Gamma ( \C^{2d} ) $. Moreover,
\begin{equation*}
    \mathbf{Z}_{s, n} f_n = \eta f_n - h_n, \qquad
    \sum_{n = 0}^{\infty} \norm{\mathbf{Z}_{s, n} f_n}^2
    \leq 2 \eta^2 \norm{f}^2 + 2 \norm{h}^2 < \infty.
\end{equation*}
Thus $ f $ belongs to the generator domain and solves $ ( \eta \1 - \lambda ( \mathbf{Z}_s ) ) f = h $. Consequently,
\begin{equation*}
    U_s^* R_{\eta, s} U_s = \bigoplus_{n = 0}^{\infty} ( \eta \1 - \mathbf{Z}_{s, n} ) ^{ - 1}.
\end{equation*}

Recall that $ Q_{\leq r} $ from \eqref{eq-chaos-tail} projects onto $ \bigoplus_{n = 0}^r ( \C^{2d} ) ^{ \odot n} $. Multiplication by $ \1 - Q_{\leq r} $ therefore removes the first $ r + 1 $ particle components:
\begin{equation*}
    U_s^* R_{\eta, s} U_s ( \1 - Q_{\leq r} )
    = \underbrace{0 \oplus \cdots \oplus 0}_{n = 0, \ldots, r}
    \oplus \bigoplus_{n = r + 1}^{\infty} ( \eta \1 - \mathbf{Z}_{s, n} ) ^{ - 1}.
\end{equation*}
The norm of an orthogonal direct sum is the supremum of its block norms. The inverse bounds above therefore imply
\begin{equation}
    \label{eq-gap-resolvent-tail}
    \norm{U_s^* R_{\eta, s} U_s ( \1 - Q_{\leq r} )}
    = \sup_{n \geq r + 1} \norm{( \eta \1 - \mathbf{Z}_{s, n} ) ^{ - 1}}
    \leq \frac{1}{\eta + ( r + 1 ) g_s} \longrightarrow 0, \qquad r \to \infty.
\end{equation}
Here the convergence uses $ g_s > 0 $. Since each particle subspace is finite dimensional, $ Q_{\leq r} $ has finite rank, and so does $ U_s^* R_{\eta, s} U_s Q_{\leq r} $. Subtracting this truncation from $ U_s^* R_{\eta, s} U_s $ gives precisely the tail in \eqref{eq-gap-resolvent-tail}. Hence the conjugated resolvent is a limit in operator norm of finite-rank operators and is compact. Conjugating back, the finite-rank operators $ R_{\eta, s} U_s Q_{\leq r} U_s^* $ approximate $ R_{\eta, s} $ with the same bound on the norm of the difference, proving the assertion. Finally, $ ( \eta \1 - ( L^{ ( s ) } ) ^* ) ^{ - 1} = R_{\eta, s}^* $ is compact because the adjoint of a compact operator is compact.
\end{proof}

\subsection{Examples separating compactness from spectral gaps}

The first example has a unique faithful normal invariant state and immediate compactness, but zero coercive spectral gap for every embedding. The second has nonunique faithful normal invariant states and a positive gap off the full kernel, while its induced semigroups and generator resolvents are noncompact. The examples therefore address both the converse implication and the role of uniqueness.

\begin{example}[Immediate compactness without a spectral gap]\label{example-coupled-modes}
We use the two-mode example in \cite[Example~3.9]{li2026kms}. Choose $ v > u > 0 $ and put
\begin{equation*}
    \nu = \frac{v^2 + u^2}{v^2 - u^2}, \qquad
    E = \begin{bmatrix} 1 & 1 \\ 1 & 1 \end{bmatrix}, \qquad
    \Omega = \begin{bmatrix} \omega_1 & h \\ h & \omega_2 \end{bmatrix},
\end{equation*}
where $ \omega_1, \omega_2, h \in \R $ and $ \omega_1 \neq \omega_2 $. There is one noise operator, and the Hamiltonian contains a quadratic pairing term:
\begin{equation}
    \label{eq-example-H-L}
    \begin{aligned}
        L_1 & = v ( a_1 + a_2 ) + u ( a_1^\dagger + a_2^\dagger ),\\
        H & = \sum_{j, k = 1}^2 \Omega_{jk} a_j^\dagger a_k
        + \frac{\mi u v}{2 \nu} \left( ( a_1^\dagger + a_2^\dagger ) ^2 - ( a_1 + a_2 ) ^2 \right).
    \end{aligned}
\end{equation}
The corresponding formal GKSL expression is
\begin{equation*}
    \cL ( X ) = \mi [ H, X ] + L_1^* X L_1
    - \frac{1}{2} L_1^* L_1 X - \frac{1}{2} X L_1^* L_1.
\end{equation*}
Set $ \alpha_\pm = ( v^2 - u^2 ) /2 \pm u v/\nu $. Both coefficients are positive, since $ v^2 + u^2 > 2 u v $. The drift and diffusion matrices are
\begin{equation}
    \label{eq-example-drift}
    \mathbf{Z} = \begin{bmatrix} - \alpha_+ E & - \Omega \\ \Omega & - \alpha_- E \end{bmatrix}, \qquad
    \mathbf{C} = \begin{bmatrix} ( v + u ) ^2 E & 0 \\ 0 & ( v - u ) ^2 E \end{bmatrix}.
\end{equation}
The symmetric part of the drift satisfies
\begin{equation}
    \label{eq-example-symmetric-part}
    \frac{\mathbf{Z} + \mathbf{Z}^*}{2}
    = - \begin{bmatrix} \alpha_+ E & 0 \\ 0 & \alpha_- E \end{bmatrix}, \qquad
    \ker ( \mathbf{Z} + \mathbf{Z}^* )
    = \myspan \{ ( 1, - 1, 0, 0 ) ^T, ( 0, 0, 1, - 1 ) ^T \}.
\end{equation}
Every eigenvalue of $ \mathbf{Z} $ has nonpositive real part. If $ \mathbf{Z} \xi = \lambda \xi $ with $ \Re \lambda = 0 $, then \eqref{eq-example-symmetric-part} forces $ \xi = ( x, - x, y, - y ) ^T $. Requiring $ \mathbf{Z} \xi $ to belong to the same kernel gives $ ( \omega_1 - \omega_2 ) x = ( \omega_1 - \omega_2 ) y = 0 $. Thus $ \xi = 0 $, a contradiction. Hence $ \mathbf{Z} $ is stable.

Put $ \beta = \log ( v^2/u^2 ) $. The centered Gaussian state with covariance $ \mathbf{S} = \nu \1_4 $ has density
\begin{equation}
    \label{eq-example-rho}
    \rho = \rho_{\beta} \otimes \rho_{\beta}.
\end{equation}
It is faithful because $ \beta > 0 $, and $ \mathbf{Z}^T \mathbf{S} + \mathbf{S} \mathbf{Z} = - \mathbf{C} $ proves invariance by \eqref{eq-gaussian-action}. Theorem~\ref{theorem-main} therefore gives uniqueness of the normal invariant state, immediate trace-class membership of every induced semigroup, and compact resolvent of every induced generator. In the KMS embedding, the common inverse temperature makes $ \mathbf{S}_{1/2} $ scalar, so $ \mathbf{Z}_{1/2} = \mathbf{Z} $. The nontrivial kernel in \eqref{eq-example-symmetric-part} shows that the KMS spectral gap vanishes, as established in \cite[Example~3.9]{li2026kms}.

In fact, the coercive spectral gap vanishes for every embedding. By Lemma~\ref{lemma-explicit-metric}, $ \mathbf{S}_s = c_{s, 1} \1_4 + \mi d_{s, 1} \mathbf{J} $, with the coefficients evaluated at the common inverse temperature $ \beta $. The kernel in \eqref{eq-example-symmetric-part} is invariant under $ \mathbf{J} $, hence under $ \mathbf{S}_s^{ - 1/2} $. The skew-adjoint part of $ \mathbf{Z} $ commutes with $ \mathbf{J} $ and therefore with $ \mathbf{S}_s $. Consequently,
\begin{equation*}
    \mathbf{Z}_s \xi = \frac{\mathbf{Z} - \mathbf{Z}^*}{2} \xi, \qquad
    \xi \in \ker ( \mathbf{Z} + \mathbf{Z}^* ).
\end{equation*}
Regard any unit vector $ \xi $ in this kernel as a one-particle Fock vector. Then $ Y = U_s \xi $ belongs to $ \dom L^{ ( s ) } $, is orthogonal to $ \rho^{1/2} $, and satisfies $ \Re \langle Y, L^{ ( s ) } Y \rangle_2 = 0 $. Contractivity gives $ \gap L^{ ( s ) } \geq 0 $, so $ \gap L^{ ( s ) } = 0 $ for every $ s \in [ 0, 1 ] $. Thus immediate trace-class membership does not imply the existence of a spectral gap, even in the KMS embedding.
\end{example}

\begin{example}[A spectral gap with an infinite-dimensional kernel]\label{example-gap-infinite-kernel}
Choose $ \gamma > 0 $ and $ \beta > 0 $ independently. Let $ \mathcal{R}_t $ be the one-mode quantum Ornstein--Uhlenbeck semigroup studied by Cipriani, Fagnola, and Lindsay \cite[Eq.~(1.2) and Section~4]{cipriani2000spectral}, with drift $ - \gamma \1_2 $ and faithful thermal invariant state $ \rho_{\beta} $ from \eqref{eq-thermal-density}. Here $ \1_2 $ is the identity on the auxiliary one-particle space $ \C^2 $. The semigroup has zero Hamiltonian and noise operators
\begin{equation*}
    L_{-} = \sqrt{\frac{2 \gamma}{1 - \me^{ - \beta}}}\,a,
    \qquad L_{+} = \sqrt{\frac{2 \gamma \me^{ - \beta}}{1 - \me^{ - \beta}}}\,a^\dagger.
\end{equation*}
The squared coefficients of $ a $ and $ a^\dagger $ have difference $ 2 \gamma $ and creation-to-annihilation ratio $ \me^{ - \beta} $. Thus $ \gamma $ fixes the one-particle decay rate, while $ \beta $ fixes the inverse temperature; varying $ \beta $ at fixed $ \gamma $ changes the coefficients of both noise operators.

On two modes, set $ \cT_t = \1_{\cB ( \Gamma ( \C ) ) } \otimes \mathcal{R}_t $ and choose the faithful invariant density $ \rho = \rho_{\beta} \otimes \rho_{\beta} $. The first factor leaves every observable unchanged, so every faithful first-mode density tensored with $ \rho_{\beta} $ is invariant. In particular, the invariant state is not unique. Fix $ s \in [ 0, 1 ] $, and let $ G^{ ( s ) } $ denote the generator induced by $ \mathcal{R}_t $ with invariant density $ \rho_{\beta} $. Identify the two-mode Hilbert--Schmidt space with $ \cB_2 ( \Gamma ( \C ) ) \otimes \cB_2 ( \Gamma ( \C ) ) $. For $ A, B \in \cB_2 ( \Gamma ( \C ) ) $, the product embedding gives
\begin{equation*}
    T_t^{ ( s ) } ( A \otimes B ) = A \otimes \me^{t G^{ ( s ) }} B, \qquad t \geq 0.
\end{equation*}
For $ B \in \dom G^{ ( s ) } $, differentiation at $ t = 0 $ yields
\begin{equation*}
    L^{ ( s ) } ( A \otimes B )
    = \lim_{t \downarrow 0} A \otimes \frac{\me^{t G^{ ( s ) }} B - B}{t}
    = A \otimes G^{ ( s ) } B.
\end{equation*}
The first-mode identity semigroup has generator zero. The product rule therefore gives
\begin{equation*}
    L^{ ( s ) } = 0 \otimes \1 + \1 \otimes G^{ ( s ) } = \1 \otimes G^{ ( s ) }.
\end{equation*}
Here each $ \1 $ acts on the corresponding Hilbert--Schmidt factor, and the operator identity denotes the closure of the displayed action on elementary tensors with $ B \in \dom G^{ ( s ) } $.

Let $ U_{s, \beta} : \Gamma ( \C^2 ) \to \cB_2 ( \Gamma ( \C ) ) $ be the one-mode unitary from Theorem~\ref{theorem-all-s-representation} associated with $ \rho_{\beta} $. The scalar drift is unchanged by the similarity transform, so
\begin{equation*}
    U_{s, \beta}^* G^{ ( s ) } U_{s, \beta} = - \gamma \lambda ( \1_2 ), \qquad U_{s, \beta} e ( 0 ) = \rho_{\beta}^{1/2}.
\end{equation*}
For an arbitrary two-mode Hilbert--Schmidt vector $ Y $, apply this unitary identification only in the second factor and write
\begin{equation*}
    f := ( \1 \otimes U_{s, \beta}^* ) Y = \bigoplus_{n = 0}^{\infty} f_n, \qquad
    f_n \in \cB_2 ( \Gamma ( \C ) ) \otimes ( \C^2 ) ^{ \odot n}.
\end{equation*}
Here $ f_n $ is the component with $ n $ particles in the thermal factor; the first factor remains unrestricted. Unitarity gives $ \norm{Y}_2^2 = \norm{f}^2 = \sum_{n = 0}^{\infty} \norm{f_n}^2 $. By Corollary~\ref{corollary-generator-domain}, applied in the second factor, $ Y \in \dom L^{ ( s ) } $ if and only if $ \sum_{n = 0}^{\infty} \gamma^2 n^2 \norm{f_n}^2 < \infty $, and then
\begin{equation*}
    ( \1 \otimes U_{s, \beta}^* ) L^{ ( s ) } Y = \bigoplus_{n = 0}^{\infty} ( - \gamma n ) f_n.
\end{equation*}
The kernel corresponds to the $ n = 0 $ subspace. Since $ U_{s, \beta} e ( 0 ) = \rho_{\beta}^{1/2} $, this gives $ \ker L^{ ( s ) } = \cB_2 ( \Gamma ( \C ) ) \otimes \C \rho_{\beta}^{1/2} $, and $ Y \perp \ker L^{ ( s ) } $ is equivalent to $ f_0 = 0 $. For every such generator-domain vector, the unitarity of $ \1 \otimes U_{s, \beta}^* $ and the orthogonality of the particle components give
\begin{equation*}
    \langle Y, L^{ ( s ) } Y \rangle_2 = \left\langle ( \1 \otimes U_{s, \beta}^* ) Y, ( \1 \otimes U_{s, \beta}^* ) L^{ ( s ) } Y \right\rangle = \sum_{n = 0}^{\infty} \langle f_n, - \gamma n f_n \rangle = - \gamma \sum_{n \geq 1} n \norm{f_n}^2.
\end{equation*}
Taking the negative real part and using $ f_0 = 0 $ therefore yields
\begin{equation*}
    - \Re \langle Y, L^{ ( s ) } Y \rangle_2
    = \gamma \sum_{n \geq 1} n \norm{f_n}^2
    \geq \gamma \sum_{n \geq 1} \norm{f_n}^2
    = \gamma \norm{Y}_2^2.
\end{equation*}
The series converges by the domain condition and Cauchy--Schwarz. Conversely, choose a unit vector $ f $ supported only in the thermal $ n = 1 $ subspace and set $ Y = ( \1 \otimes U_{s, \beta} ) f $. Then $ Y $ belongs to the domain, is orthogonal to the kernel, and satisfies $ L^{ ( s ) } Y = - \gamma Y $. It therefore attains the value $ - \Re \langle Y, L^{ ( s ) } Y \rangle_2 = \gamma $. The definition \eqref{eq-definition-coercive-gap} now gives
\begin{equation}
    \label{eq-infinite-kernel-gap}
    \ker L^{ ( s ) } = \cB_2 ( \Gamma ( \C ) ) \otimes \C \rho_{\beta}^{1/2},
    \qquad \gap L^{ ( s ) } = \gamma.
\end{equation}
The stationary directions contributed by the first factor all belong to the full kernel and are excluded when computing this gap.

Finally, on the infinite-dimensional kernel, for $ \eta > 0 $ and $ t > 0 $,
\begin{equation*}
    \left.( \eta \1 - L^{ ( s ) } ) ^{ - 1}\right|_{\ker L^{ ( s ) }}
    = \eta^{ - 1} \1_{\ker L^{ ( s ) }}, \qquad
    \left.T_t^{ ( s ) }\right|_{\ker L^{ ( s ) }} = \1_{\ker L^{ ( s ) }}.
\end{equation*}
Neither restriction is compact, since a nonzero scalar multiple of the identity on an infinite-dimensional space is not compact. Thus a positive coercive spectral gap on the orthogonal complement of the full kernel does not imply compact resolvent or compactness at positive times without uniqueness of the invariant state.
\end{example}

\section*{Use of artificial intelligence tools}

During the preparation of this manuscript, the author used ChatGPT to assist with language editing, organization of the exposition, and exploratory discussion of mathematical arguments. The author formulated the mathematical content, wrote and revised the manuscript, and takes full responsibility for its content.


\begin{thebibliography}{LWW26}
\setlength{\parsep}{0pt}
\setlength{\itemsep}{0pt}
\renewcommand{\makelabel}[1]{#1\hss}
\bibitem[AFP22]{agredo2022decoherence}
J.~Agredo, F.~Fagnola, and D.~Poletti.
``The decoherence-free subalgebra of {Gaussian} quantum {Markov} semigroups''.
In: \emph{Milan Journal of Mathematics} 90 (2022), pp.~257--289.
\textsc{doi}: \href{https://doi.org/10.1007/s00032-022-00355-0}{\nolinkurl{10.1007/s00032-022-00355-0}}.


\bibitem[ASS95]{ArakiSinhaSunder1995}
H.~Araki, K.~B. Sinha, and V.~S. Sunder.
``On boundedness and $ \norm{\cdot}_p $ continuity of second quantisation''.
In: \emph{Publications of the Research Institute for Mathematical Sciences} 31.5 (1995), pp.~941--952.
\textsc{doi}: \href{https://doi.org/10.2977/PRIMS/1195163724}{\nolinkurl{10.2977/PRIMS/1195163724}}.


\bibitem[CF00]{carbone2000exponential}
R.~Carbone and F.~Fagnola.
``Exponential $ L^2 $-convergence of quantum {Markov} semigroups on $ \mathcal{B} ( h ) $''.
In: \emph{Mathematical Notes} 68.4 (2000), pp.~452--463.
\textsc{doi}: \href{https://doi.org/10.1007/BF02676724}{\nolinkurl{10.1007/BF02676724}}.


\bibitem[CFL00]{cipriani2000spectral}
F.~Cipriani, F.~Fagnola, and J.~M. Lindsay.
``Spectral analysis and {Feller} property for quantum {Ornstein--Uhlenbeck} semigroups''.
In: \emph{Communications in Mathematical Physics} 210.1 (2000), pp.~85--105.
\textsc{doi}: \href{https://doi.org/10.1007/s002200050773}{\nolinkurl{10.1007/s002200050773}}.


\bibitem[CG02]{ChojnowskaMichalikGoldys2002}
A.~Chojnowska-Michalik and B.~Goldys.
``Symmetric {Ornstein--Uhlenbeck} semigroups and their generators''.
In: \emph{Probability Theory and Related Fields} 124.4 (2002), pp.~459--486.
\textsc{doi}: \href{https://doi.org/10.1007/s004400200222}{\nolinkurl{10.1007/s004400200222}}.


\bibitem[CG96]{ChojnowskaMichalikGoldys1996}
A.~Chojnowska-Michalik and B.~Goldys.
``Nonsymmetric {Ornstein--Uhlenbeck} semigroup as second quantized operator''.
In: \emph{Journal of Mathematics of Kyoto University} 36.3 (1996), pp.~481--498.
\textsc{doi}: \href{https://doi.org/10.1215/kjm/1250518505}{\nolinkurl{10.1215/kjm/1250518505}}.


\bibitem[EN00]{EngelNagel2000}
K.-J.~Engel and R.~Nagel.
\emph{One-Parameter Semigroups for Linear Evolution Equations}.
Graduate Texts in Mathematics~194.
New York: Springer, 2000.
\textsc{doi}: \href{https://doi.org/10.1007/b97696}{\nolinkurl{10.1007/b97696}}.

\bibitem[EL77]{EvansLewis1977}
D.~E. Evans and J.~T. Lewis.
\emph{Dilations of Irreversible Evolutions in Algebraic Quantum Theory}.
Communications of the Dublin Institute for Advanced Studies, Series A: Theoretical Physics, no.~24.
Dublin: Dublin Institute for Advanced Studies, 1977.
\textsc{url}: \url{https://orca.cardiff.ac.uk/id/eprint/34031/}.


\bibitem[FL25]{fagnola_li_2025spectral}
F.~Fagnola and Z.~Li.
\emph{Spectral analysis for {Gaussian} quantum {Markov} semigroups}.
2025.
arXiv: \href{https://arxiv.org/abs/2504.12162v2}{\nolinkurl{2504.12162v2}} (revised 2026).


\bibitem[FPSU25]{fagnola2025spectral}
F.~Fagnola, D.~Poletti, E.~Sasso, and V.~Umanit\`{a}.
``The spectral gap of a Gaussian quantum Markovian generator''.
In: \emph{Journal of Functional Analysis} 289.10 (2025), p.~111119.
\textsc{doi}: \href{https://doi.org/10.1016/j.jfa.2025.111119}{\nolinkurl{10.1016/j.jfa.2025.111119}}.

\bibitem[GGM04]{GeorgescuGerardMoller2004}
V.~Georgescu, C.~G\'{e}rard, and J.~S. M{\o}ller.
``Spectral theory of massless {Pauli--Fierz} models''.
In: \emph{Communications in Mathematical Physics} 249.1 (2004), pp.~29--78.
\textsc{doi}: \href{https://doi.org/10.1007/s00220-004-1111-x}{\nolinkurl{10.1007/s00220-004-1111-x}}.
Author manuscript: \href{https://www.imo.universite-paris-saclay.fr/~christian.gerard/publis/nelson-massless.pdf}{April 2003 version}.

\bibitem[GP26]{girotti2025gaussian}
F.~Girotti and D.~Poletti.
``{Gaussian} quantum {Markov} semigroups on finitely many modes admitting a normal invariant state''.
In: \emph{Journal of Mathematical Analysis and Applications} 556.1 (2026), p.~130150.
\textsc{doi}: \href{https://doi.org/10.1016/j.jmaa.2025.130150}{\nolinkurl{10.1016/j.jmaa.2025.130150}}.


\bibitem[KP04]{ko2004construction}
C.~K. Ko and Y.~M. Park.
``Construction of a family of quantum {Ornstein--Uhlenbeck} semigroups''.
In: \emph{Journal of Mathematical Physics} 45.2 (2004), pp.~609--627.
\textsc{doi}: \href{https://doi.org/10.1063/1.1641150}{\nolinkurl{10.1063/1.1641150}}.


\bibitem[Li26]{li2026kms}
Z.~Li.
\emph{On the existence of the {KMS} spectral gap in {Gaussian} quantum {Markov} semigroups}.
2026.
arXiv: \href{https://arxiv.org/abs/2512.23414v2}{\nolinkurl{2512.23414v2}}.


\bibitem[LWW26]{LiuWanWu2026}
Z.~Liu, J.~Wan, and J.~Wu.
\emph{Entropy contraction and hypercontractivity for {Gaussian} quantum {Markov} semigroups}.
2026.
arXiv: \href{https://arxiv.org/abs/2610.06201v1}{\nolinkurl{2610.06201v1}}.


\bibitem[Par92]{parthasarathy1992quantum}
K.~R. Parthasarathy.
\emph{An Introduction to Quantum Stochastic Calculus}.
Monographs in Mathematics~85.
Basel: Birkh\"{a}user, 1992.
\textsc{doi}: \href{https://doi.org/10.1007/978-3-0348-8641-3}{\nolinkurl{10.1007/978-3-0348-8641-3}}.

\bibitem[Par22]{parthasarathy2021pedagogical}
K.~R. Parthasarathy.
``A pedagogical note on the computation of relative entropy of two $ n $-mode Gaussian states''.
In: \emph{Infinite Dimensional Analysis, Quantum Probability and Applications}.
Vol.~390. Springer Proceedings in Mathematics \& Statistics.
Cham: Springer, 2022, pp.~55--72.
\textsc{doi}: \href{https://doi.org/10.1007/978-3-031-06170-7_2}{\nolinkurl{10.1007/978-3-031-06170-7_2}}.
arXiv: \href{https://arxiv.org/abs/2102.06708v2}{\nolinkurl{2102.06708v2}}.

\bibitem[Sim05]{Simon2005}
B.~Simon.
\emph{Trace Ideals and Their Applications}.
2nd ed. Vol.~120. Mathematical Surveys and Monographs.
Providence, RI: American Mathematical Society, 2005.
\textsc{doi}: \href{https://doi.org/10.1090/surv/120}{\nolinkurl{10.1090/surv/120}}.

\bibitem[vN05]{vanNeerven2005}
J.~M.~A.~M. van Neerven.
``Second quantization and the $ L^p $-spectrum of nonsymmetric {Ornstein--Uhlenbeck} operators''.
In: \emph{Infinite Dimensional Analysis, Quantum Probability and Related Topics} 8.3 (2005), pp.~473--495.
\textsc{doi}: \href{https://doi.org/10.1142/S0219025705002074}{\nolinkurl{10.1142/S0219025705002074}}.

\bibitem[Wir26]{Wirth2026}
M.~Wirth.
\emph{The {KMS} and {GNS} spectral gap of quantum {Markov} semigroups}.
2026.
arXiv: \href{https://arxiv.org/abs/2604.21630v1}{\nolinkurl{2604.21630v1}}.


\end{thebibliography}
\end{document}